\documentclass{amsart}
\usepackage[latin1]{inputenc}
\usepackage[T1]{fontenc} 
\usepackage{amsmath,amsthm,amssymb,mathrsfs}
\usepackage{times}
\usepackage{enumerate} 
\usepackage{amsfonts}
\usepackage{calligra}
\usepackage{appendix}
\usepackage{graphicx}
\usepackage{subfig}
\usepackage[english]{babel}
\usepackage{mathtools}
\usepackage{tensor}
\usepackage{graphicx}
\usepackage{subfig}
\usepackage{color}
\usepackage{hyperref}

\def\R{\mathbb{R}}
\def\N{\mathbb{N}}
\def\U{\mathcal{U}}
\def\V{\mathcal{V}}
\def\T{\mathbb{T}}
\def\supp{\operatorname{supp}}

\newcommand{\norm}[1]{\left\lVert#1\right\rVert}
\newcommand{\abs}[1]{\left\lvert#1\right\rvert}
\def\spann{\operatorname{span}}

\newcommand{\defeq}{\mathrel{:\mkern-0.25mu=}}
\newcommand{\eqdef}{\mathrel{=\mkern-0.25mu:}}

\newtheorem{thm}{Theorem}[section]
\newtheorem{cor}[thm]{Corollary}
\newtheorem{prop}[thm]{Proposition}
\newtheorem{lem}[thm]{Lemma}

\theoremstyle{definition}
\newtheorem{defin}[thm]{Definition}
\newtheorem{rem}[thm]{Remark}

\numberwithin{equation}{section}

\begin{document}

\title{Bounded solutions of ideal MHD with compact support in space-time}


\date{}

\author{Daniel Faraco}
\address{Departamento de Matem\'{a}ticas \\ Universidad Aut\'{o}noma de Madrid, E-28049 Madrid, Spain; ICMAT CSIC-UAM-UC3M-UCM, E-28049 Madrid, Spain}
\email{daniel.faraco@uam.es}
\thanks{D.F. was partially supported by ICMAT Severo Ochoa projects SEV-2011-0087 and SEV-2015-556, the grants MTM2014-57769-P-1 and MTM2017-85934-C3-2-P (Spain) and the ERC grant 307179-GFTIPFD, ERC grant 834728-QUAMAP.  
S.L. was supported by the ERC grant 307179-GFTIPFD and by the AtMath Collaboration at the University of Helsinki. L. Sz. was supported by ERC Grant 724298-DIFFINCL. Part of this work was completed in the Hausdorff Research Institute (HIM) in Bonn during the Trimester Programme Evolution of Interfaces. The authors gratefully acknowledge the warm hospitality of HIM during this time. D.F. also thanks the hospitality of the University of Aalto where part of his research took place.}

\author{Sauli Lindberg}
\address{Department of Mathematics and Statistics \\ University of Helsinki, P.O. Box 68, 00014 Helsingin yliopisto, Finland}
\email{sauli.lindberg@helsinki.fi}

\author{L\'{a}szl\'{o} Sz\'{e}kelyhidi, Jr.}
\address{Institut f\"{u}r mathematik \\ Universit\"{a}t Leipzig, Augustusplatz 10, D-04109, Leipzig,
Germany}
\email{laszlo.szekelyhidi@math.uni-leipzig.de}

\date{}

\begin{abstract}
We show that in 3-dimensional ideal magnetohydrodynamics there exist infinitely many bounded solutions that are compactly supported in space-time and have non-trivial velocity and magnetic fields. The solutions violate conservation of total energy and cross helicity, but preserve magnetic helicity. For the 2-dimensional case we show that, in contrast, no nontrivial compactly supported solutions exist in the energy space.
\end{abstract}

\maketitle

\section{Introduction}
Ideal magnetohydrodynamics (MHD in short) couples Maxwell equations with Euler equations to study the macroscopic behaviour of electrically conducting fluids such as plasmas and liquid metals (see~\cite{GLBL} and~\cite{ST}). The corresponding system of partial differential equations governs the simultaneous evolution of a velocity field $u$ and a magnetic field $B$ which are divergence free. The evolution of $u$ is described by the Cauchy momentum equation with an external force given by the Lorentz force induced by $B$. The evolution of $B$, in turn, is described by the induction equation which couples Maxwell-Faraday law with Ohm's law. 



The ideal MHD equations give a wealth of structure to smooth solutions and several integral quantities are preserved.
In 3D, smooth solutions conserve the \emph{total energy}, but also two other quantities related to the topological invariants of the system are constant functions of time: The \emph{cross helicity}  measures the entanglement of vorticity  and magnetic field  and the \emph{magnetic helicity} measures the linkage and twist  of magnetic field lines. 
Magnetic helicity was first studied by Woltjer in~\cite{Woltjer} and interpreted topologically in the highly influential work of Moffatt~\cite{Mof}, see also~\cite{ArnoldKhesinbook}. In fact, it was recently been proved in \cite{KPY} that cross helicity and magnetic helicity characterise all regular integral invariants of ideal MHD.

In this paper we are interested in weak solutions 
of the ideal MHD system, which in some sense describe the infinite Reynolds number limit. As pointed out in \cite{CKS} such weak solutions should reflect two properties:
\begin{enumerate}
\item[(i)] anomalous dissipation of energy;
\item[(ii)] conservation of magnetic helicity.
\end{enumerate}
Indeed, just as in the hydrodynamic situation, in MHD turbulence the rate of total energy dissipation in viscous, resistive MHD seems not to tend to zero when the Reynolds number and magnetic Reynolds number tend to infinity. This has been recently verified numerically in 3D, see~\cite{DA},~\cite{LBMM} and~\cite{MP}. On the other hand simulations and theoretical results have shown that magnetic helicity is a rather robust conserved quantity even in turbulent regimes, and J.B.Taylor conjectured that  magnetic helicity is approximately conserved for small resistivities~\cite{Taylor} (unlike subhelicities along Lagrangian subdomains that are magnetically closed at the initial time). Taylor's conjecture is at the core of Woltjer-Taylor relaxation theory which predicts that after an initial turbulent state, various laboratory plasmas relax towards a quiescent state which minimises magnetic energy subject to the constraint of magnetic helicity conservation (see~\cite{OS,ArnoldKhesinbook}).

The conservation of magnetic helicity for weak solutions of ideal MHD was first addressed in \cite{CKS}, and subsequently it was shown in~\cite{Aluie} and~\cite{KL} that it is conserved if $u,B \in L^3_{x,t}$, i.e., in contrast with energy conservation, no smoothess is required. Moreover, the first and second author recently proved that if a solution in the energy space $L^\infty_t L^2_x$ arises as an inviscid limit, then it conserves magnetic helicity (see~\cite{FL}). In this context, Theorem \ref{Magnetic helicity preservation theorem}  below extends~\cite[Corollary 1.3]{FL}, from ideal (i.e. inviscid, non-resistive) limits of Leray-Hopf solutions  to a larger class of possible approximation schemes. 

Our main purpose in this paper is to show the existence of nontrivial weak solutions to ideal 3D MHD compatible with both requirements (i)-(ii) above:


\begin{thm} \label{MHD 3D theorem}
There exist bounded, compactly supported weak solutions of ideal MHD in $\R^3$, with both $u, B$ nontrivial, such that neither total energy nor cross helicity is conserved in time.
\end{thm}

We note that bounded solutions in particular fall into the subcritical regime of \cite{KL} for magnetic helicity, so that for the solutions above magnetic helicity must vanish at all times even though the magnetic field $B$ is not identically zero. Moreover, 
as a corollary of Theorem \ref{Magnetic helicity preservation theorem} below, it also holds that for bounded solutions on $\T^3$, either the initial data has vanishing magnetic helicity or $B$ cannot have compact support in time. Indeed, as noted by Arnold~\cite{Arn}, $\int_{\T^3} \abs{B}^2 dx \ge C \abs{\int_{\T^3} A \cdot B dx}$ at every $t \in [0,T[$, where $A$ is the magnetic potential. 
It is also worth pointing out that the solutions in Theorem \ref{MHD 3D theorem} have nontrivial cross-helicity. 

MHD turbulence in 2D seems to have many similarities with the 3D case (in stark contrast with hydrodynamic turbulence), in particular there is plenty of numerical evidence for anomalous dissipation of energy \cite{BiskampWelter,BiskampBook}. Nevertheless, we will show in Section 2.2 that in 2D, under very mild conditions, weak solutions with nontrivial magnetic field cannot decay to zero in finite time, in particular solutions as in Theorem 1.1 do not exist in 2D.

Our construction is based on the framework developed in \cite{DLS09} by C.~De Lellis and the third author for the construction of weak solutions to the Euler equations. This framework is based on convex integration, which was developed by Gromov \cite{Gromov} following the work of Nash \cite{Nash}, and -- in a nutshell -- amounts to an iteration procedure whereby one approximates weak solutions via a sequence of subsolutions, in each iteration adding highly oscillatory perturbations designed to cancel the low wavenumber part of the error. In \cite{DLS09} convex integration was used in connection with Tartar's framework to obtain bounded nontrivial weak solutions of the Euler equations which have compact support and violate energy conservation. Such pathological weak solutions were known to exist \cite{Scheffer,Shnirelman} but the method of~\cite{DLS09} turned out to be very robust and many equations in hydrodynamics are amenable to it and its ramifications.

Roughly speaking the development of the theory followed two strands: concerning the Euler equations and in connection with Onsager's conjecture \cite{Onsager,Eyi}, an important problem was to push the regularity of such weak solutions beyond mere boundedness to the Onsager-critical regime. This programme, started in \cite{DLS13} finally culminated in Isett's work \cite{Ise16}, see also \cite{BDLSV}. For a thorough report of these developments and connections to Nash's work on isometric embeddings, we refer to \cite{DLS16}. Another, somewhat independent strand, was to adapt the techniques to other systems of equations, such as compressible Euler system \cite{CDLK}, active scalar equations \cite{CFG,Shvydkoy,Szekelyhidi,BSV} and others \cite{BLFNL, CM, CS, KY, Nov}. A key point in the technique is a study of the phase-space geometry of the underlying system, to understand the interaction of high-frequency perturbations with the nonlinearity in the equations in the spirit of L.~Tartar's compensated compactness. A particularly relevant example to this discussion is the case of 2D active scalar equations, where there seem to be a dichotomy between systems closed under weak convergence such as 2D Euler in vorticity form or SQG, and those with a large weak closure such as IPM \cite{CFG, Shvydkoy} - see the discussion in Section 8 of \cite{DLS12} and \cite{IV} in this regard. 

Concerning the ideal MHD system, setting the magnetic field $b \equiv 0$ obviously reduces to the incompressible Euler equations, and thus \cite{DLS09} applies. More generally, in~\cite{BLFNL} Bronzi, Lopes Filho and Nussenzveig Lopes constructed bounded weak solutions of the symmetry reduced form $u(x_1,x_2,x_3,t) = (u_1(x_1,x_2,t),u_2(x_1,x_2,t),0)$ and $B(x,t) = (0,0,b(x_1,x_2,t))$, compactly supported in time and not identically zero. These ``$2\frac{1}{2}$-dimensional'' solutions were obtained by reducing the symmetry reduced 3D MHD to 2D Euler with a passive tracer, where a modification of the strategy of \cite{DLS09} applied. Nevertheless, such reductions to the Euler system do not seem to be able to capture generic, truly 3-dimensional weak solutions, which -- with the simultaneous requirement of properties (i) and (ii) above -- seem to lie on the borderline between weakly closed (e.g. SQG) and non-closed (e.g. IPM) systems.

This remark will be explained in more detail in Section \ref{s:system} below - for the introduction let us merely point out that whilst the Cauchy momentum equation for the evolution of the velocity $u$ has a large relaxation (the main observation behind all results involving convex integration for the Euler equations), the Maxwell system for the evolution of the magnetic field $B$ is weakly closed (an observation going back to the pioneering work of Tartar \cite{Tartar}).Indeed, our whole philosophy in this paper is to emphasise the role of compensated compactness in connection with conserved quantities - in Section 2 we revisit Taylor's conjecture and conservation of mean-square magnetic potential conservation in this light. In turn, to deal with the additional rigidity due to conservation of magnetic helicity, inspired by~\cite{MS99} we develop a version of convex integration directly on differential two-forms (the Maxwell 2-form), consistent with the geometry of full 3D MHD.

\section{The ideal MHD system}\label{s:system}

We recall that the ideal MHD equations in three space dimensions are written as 
\begin{align}
& \partial_t u + u \cdot \nabla u- B \cdot \nabla B + \nabla \Pi = 0, \label{MHD} \\
& \partial_t B + \nabla \times (B \times u) = 0,\label{MHD2} \\
& \nabla \cdot u = \nabla \cdot B = 0, \label{MHD3}
  \end{align}
for a velocity field $u$, magnetic field $B$ and total pressure $\Pi$. In this paper we consider both the full space case $\R^3$ and the periodic setting $\T^3$.
In the latter case the zero-mean condition
\begin{equation} \label{MHD5}
\langle u\rangle = 0,\quad  \langle B\rangle = 0 \qquad \text{for a.e } t 
\end{equation}
is added to \eqref{MHD}--\eqref{MHD3}, where for notational convenience we write $\langle u\rangle$ for the spatial average on $\T^3$.

 As usual, weak solutions of  \eqref{MHD}--\eqref{MHD3} can be defined in the sense of distributions for $u,B\in L^2_{loc}$, using the identities $u \cdot \nabla u - B \cdot \nabla B = \nabla \cdot (u \otimes u - B \otimes B)$ and $\nabla\times(B\times u)=\nabla\cdot(B\otimes u-u\otimes B)$ for divergence-free fields.
That is,
\[\int_0^T \int_{\R^3} \left[ u \cdot \partial_t \varphi + (u \otimes u - B \otimes B) : D\varphi \right] + \int_{\R^3} u_0 \cdot \varphi(\cdot,0) = 0,\]
\[\int_0^T \int_{\R^3} [B \cdot \partial_t \varphi + (B \otimes u - u \otimes B) : D \varphi] + \int_{\R^3} B_0 \cdot \varphi(\cdot,0) = 0,\]
\[\int_0^T \int_{\R^3} u \cdot \nabla \varphi = \int_0^T \int_{\R^3} B \cdot \nabla \varphi = 0\]
for appropriate Cauchy data $u_0,B_0$ for all $\varphi \in C_c^\infty(\R^3 \times [0,T[)$ with $\nabla\cdot\varphi = 0$. An analogous definition is given in the periodic setting on the torus $\T^3$.

\subsection{Conserved quantities}

It is well known that there are three classically conserved quantities of ideal 3D MHD on the torus $\T^3$. For the first two, analogous definitions are available in $\R^3$.
\begin{defin} \label{Definition of integral quantities}
Let $(u,B)$ be a smooth solution of \eqref{MHD}-\eqref{MHD3} and let $A$ be a vector potential for $B$, i.e.~$\nabla \times A = B$. The \emph{total energy}, \emph{cross helicity} and \emph{magnetic helicity}  of $(u,B)$ are defined as
\begin{align*}
&\frac{1}{2} \int_{\T^3} (\abs{u(x,t)}^2+\abs{B(x,t)}^2) \, dx, \\
&\int_{\T^3} u(x,t) \cdot B(x,t) \, dx, \\
&\int_{\T^3} A(x,t) \cdot B(x,t) \, dx.
\end{align*}
\end{defin}
All three quantities defined above are conserved in time by smooth solutions. The conservation of total energy and cross helicity conservation was studied in~\cite{CKS,KL,WZ,Yu}. Conservation of the magnetic helicity was shown in~\cite{CKS} for solutions $u \in C([0,T]; B_{3,\infty}^{\alpha_1})$ and $B \in C([0,T]; B_{3,\infty}^{\alpha_2})$ with $\alpha_1 + 2 \alpha_2 > 0$. In~\cite{KL,Aluie}, magnetic helicity conservation is shown under the assumption that $u,B \in L^3(\T^3 \times ]0,T[)$. 

We note in passing that on the whole space $\R^3$ the analogous definitions of total energy and cross helicity lead to conserved quantities for square integrable solutions, but magnetic helicity is not well-defined. This boils down to the scaling properties of the function spaces in question; see Appendix A. However, for square integrable magnetic fields that are compactly supported in space, magnetic helicity is well-defined. Indeed, every $B \in L^1(\R^3)$ with $\nabla \cdot B = 0$ satisfies $\int_{\R^3} B(x) \, dx = \hat{B}(0) = 0$ since $\hat{B}(\xi) \cdot \xi/\abs{\xi} = 0$ for all $\xi \neq 0$ and $\hat{B}$ is continuous. 

Following L.~Tartar's pioneering work \cite{Tartar} one can understand the system \eqref{MHD}--\eqref{MHD3} as a coupling between linear conservation laws and a set of constitutive laws in form of pointwise constraints.  The conservation laws are
\begin{align}
& \partial_t u + \nabla \cdot S = 0, \label{Linearised MHD1} \\
& \partial_t B + \nabla \times E = 0, \label{Linearised MHD2} \\
& \nabla \cdot u = \nabla \cdot B = 0, \label{Linearised MHD3}
  \end{align}
where in 3D, $S$ is a symmetric $3\times 3$ tensor (the Cauchy stress tensor) and $E$ is a vectorfield (the electric field). Indeed, \eqref{Linearised MHD2} is simply the Maxwell-Faraday law for the electric field. In the periodic case \eqref{MHD5} is added to \eqref{Linearised MHD1}--\eqref{Linearised MHD3}. The \emph{constitutive set} is then obtained by relating the stress tensor $S$ and the electric field $E$ to velocity, magnetic field and pressure, e.g.~via the ideal Ohm's law:
\begin{equation}\label{e:K}
K \defeq \{(u,S,B,E) \colon S = u \otimes u - B \otimes B + \Pi I, \; \Pi \in \R, \; E = B \times u\}.
\end{equation}
It is easy to verify that the system \eqref{MHD}--\eqref{MHD3} can be equivalently formulated for the state variables $(u,S,B,E)$ as \eqref{Linearised MHD1}--\eqref{Linearised MHD3} together with $(u,S,B,E)(x,t)\in K$ a.e. $(x,t)$. 

Using this formulation one can easily identify the conservation of magnetic helicity as an instance of compensated compactness following the work of L.~Tartar \cite{Tartar} when applied to the Maxwell system 
\begin{equation}\label{e:maxwell}
\begin{split} 
\partial_tB+\nabla\times E&=0,\\
\nabla\cdot E&=0.
\end{split}
\end{equation}

To explain this, we recall the following generalisation of the div-curl lemma from Example 4 in \cite{Tartar}: suppose we have a sequence of magnetic and electric fields $(B_j,E_j)\rightharpoonup (B,E)$ converging weakly in $L^2_{x,t}$ and such that $\{\partial_tB_j+\nabla\times E_j\}$ and $\{\nabla\cdot E_j\}$ are in a compact subset of $H^{-1}$. Then $B_j\cdot E_j\overset{*}{\rightharpoonup} B\cdot E$ in the space of measures. In view of the constitutive law $E=B\times u$ we deduce that any reasonable approximation of \emph{bounded} weak solutions of ideal MHD leads in the limit to a solution $(u,S,B,E)$ of \eqref{Linearised MHD1}--\eqref{Linearised MHD3} with $B\cdot E=0$. That is, the state variables are constrained to the relaxed constitutive set
\begin{equation}\label{e:M}
\mathscr{M}=\{(u,S,B,E):\, B \cdot E = 0\}.
\end{equation}
In turn, perpendicularity of the electric $E$ and magnetic $B$ fields is closely related to conservation of magnetic helicity. Indeed, if $A$ is a magnetic potential (so that $\nabla\times A=B$), adapting the classical computation (e.g. \cite{BiskampBook}) shows that
\begin{equation}\label{e:maghel}
\frac{d}{dt} \int_{\T^3} A \cdot B\, dx = - 2 \int_{\T^3} B \cdot E \, dx.
\end{equation}
More generally, we have the following theorem, establishing the connection between compensated compactness and conservation of magnetic helicity, an issue that has been emphasised by L.Tartar \cite{Tartar,Tartar2}:

\begin{thm} \label{Magnetic helicity preservation theorem}\hfill
\begin{enumerate}
\item[(a)] Suppose that $(B,E)\in L^p\times L^{p'}(\T^3 \times ]0,T[)$, $\frac{1}{p}+\frac{1}{p'}=1$, with $\langle B\rangle=0$, is a solution of \eqref{e:maxwell}
and assume $B\cdot E=0$ a.e.. Then magnetic helicity is conserved.
\item[(b)] Suppose that $(B_j,E_j)$ is a sequence of solutions of \eqref{e:maxwell} as in (a), and in addition
$$
B_j \rightharpoonup B\textrm{ in }L^2(\T^3 \times ]0,T[)\textrm{ and }\sup_{j \in \N} \norm{E_j}_{L^1(\T^3 \times ]0,T[)} < \infty.
$$ 
Then magnetic helicity is conserved.
\end{enumerate}
\end{thm}

Part (a) extends in particular the $L^3$ result of \cite{KL}. Indeed, for weak solutions of ideal MHD with $u,B\in L^3$ we have $E=B\times u\in L^{3/2}$. On the other hand our proof does not rely on a specific regularisation technique as in \cite{CET}, and merely relies on a weak version of formula \eqref{e:maghel}. Part (b) shows that conservation of magnetic helicity holds even beyond the setting of weak continuity of the quantity $B\cdot E$. As a matter of fact this line of argument furnishes a proof of Taylor conjecture for simply connected domains \cite{FL}.
 
We begin the proof of Theorem \ref{Magnetic helicity preservation theorem} by recalling the following $L^p$ Poincar\'e-type lemma for the Maxwell system \eqref{e:maxwell}:
\begin{lem} \label{Time evolution of vector potential}
Let $1 < p < \infty$, $1/p+1/p' = 1$ and suppose $(B,E)\in L^p\times L^{p'}(\T^3 \times ]0,T[)$ is a solution of \eqref{e:maxwell}.
Then there exist a unique $A \in L^p_tW^{1,p}_x(\T^3 \times ]0,T[)$ and $\varphi \in L^{p'}_t W^{1,p'}_x(\T^3 \times ]0,T[)$ such that
\[
B = \nabla \times A \quad \text{and} \quad \partial_t A + E - \langle E\rangle = \nabla \varphi
\]
with $\langle A\rangle = 0$, $\langle\varphi\rangle = 0$ for a.e. $t$ and $\nabla \cdot A = 0$. Furthermore,
\[\norm{\nabla A}_{L^p} \lesssim \norm{B}_{L^p} \quad \text{and} \quad \norm{\partial_t A}_{L^{p'}} + \norm{\nabla \varphi}_{L^{p'}} \lesssim \norm{E}_{L^{p'}}.\]
\end{lem}
Indeed, we set $A=-\Delta^{-1}(\nabla\times B)$ and $\varphi=\Delta^{-1}\nabla\cdot(\partial_tA+E)$ and apply standard Calder\'on-Zygmund estimates for the Laplacian.

\begin{proof}[Proof of Theorem \ref{Magnetic helicity preservation theorem}]
For part (a) suppose $(B,E) \in L^p\times L^{p'}(\T^3 \times ]0,T[)$ is a solution of \eqref{e:maxwell} with $B\cdot E=0$. Let $\eta \in C_c^\infty(]0,T[)$, so that for $\varepsilon > 0$ small enough, $\supp(\eta) \subset ]\varepsilon,T-\varepsilon[$. Furthermore, let $B_\delta=B*\chi_\delta$ be a standard space-time mollification of $B$. By using Lemma \ref{Time evolution of vector potential} and integrating by parts a few times we get
\[\begin{array}{lcl}
& & \displaystyle \int_\varepsilon^{T-\varepsilon} \partial_t \eta(t) \int_{\T^3} A(x,t) \cdot B(x,t) \, dx \, dt \\
&=& \displaystyle \lim_{\delta \searrow 0} \int_\varepsilon^{T-\varepsilon} \partial_t \eta(t) \int_{\T^3} A_\delta(x,t) \cdot B_\delta(x,t) \, dx \, dt \\
&=& \displaystyle \lim_{\delta \searrow 0} \left[ \int_\varepsilon^{T-\varepsilon} \eta(t) \int_{\T^3} \bigl( E_\delta(x,t) - \langle E\rangle - \nabla \varphi_\delta(x,t) \bigr) \cdot B_\delta(x,t) \, dx \, dt \right. \\
&+& \displaystyle \left. \int_\varepsilon^{T-\varepsilon} \eta(t) \int_{\T^3} A_\delta(x,t) \cdot \nabla \times E_\delta(x,t) \, dx \, dt \right] \\
&=& \displaystyle 2 \lim_{\delta \searrow 0} \int_\varepsilon^{T-\varepsilon} \eta(t) \int_{\T^3} E_\delta(x,t) \cdot B_\delta(x,t) \, dx \, dt \\
&=& \displaystyle 2 \int_\varepsilon^{T-\varepsilon} \eta(t) \int_{\T^3} B(x,t) \cdot E(x,t) \, dx \, dt = 0,
\end{array}\]
since $|B||E|\in L^1(\T^3 \times ]0,T[)$ and $B \cdot E = 0$.

\medskip

For part (b) suppose $(B_j,E_j) \in L^p\times L^{p'}(\T^3 \times ]0,T[)$ is a sequence of solutions of \eqref{e:maxwell} with $B_j\cdot E_j=0$ a.e. and assume $B_j \rightharpoonup B$ in $L^2(\T^3 \times ]0,T[)$ and $\sup_{j \in \N} \norm{E_j}_{L^1} < \infty$. We intend to use to the Aubin-Lions Lemma (see e.g.~\cite[Lemma 7.7]{Rou}) to get, up to a subsequence, $A_j \to A$ in $L^2_t L^2_x(\T^3 \times ]0,T[;\R^3)$; then $\nabla \times A = B$, $\nabla \cdot A = 0$ and $\langle A\rangle$ a.e. $t$, and furthermore $\int_{\T^3} A \cdot B\, dx$ is constant in $t$.

First note that $\sup_{j \in \N} \norm{A_j}_{L^2_t W^{1,2}_x} < \infty$. Let us denote $W^{1,4}_\sigma(\T^3;\R^3) \defeq \{w \in W^{1,4}(\T^4;\R^3) \colon \nabla \cdot w = 0\}$. By using the embedding $L^1(\T^3,\R^3) \hookrightarrow (W^{1,4}_\sigma(\T^3,\R^3))^*$ and the formula $\partial_t A_j = - E_j + \langle E\rangle + \nabla \varphi_j$ we obtain
\[\sup_{j \in \N} \norm{\partial_t A_j}_{L^1_t (W^{1,4}_\sigma)^*_x} \lesssim \sup_{j \in \N} \norm{E_j}_{L^1} < \infty,\]
which verifies the assumptions of the Aubin-Lions Lemma.
\end{proof}

\subsection{The 2-dimensional case}

In comparison to the above analysis, let us briefly look at the 2-dimensional case. Here \eqref{MHD2} reduces to
\begin{equation}\label{MHD2-2D}
\partial_t B +\nabla^\perp (u\cdot B^\perp) = 0, 
\end{equation}
where we write $B^\perp=(-B_2,B_1)$ for the vector $B=(B_1,B_2)$, and similarly $\nabla^\perp=(-\partial_2,\partial_1)$. The magnetic potential (stream function) of $B$ is a scalar field $\psi$ such that $\nabla^\perp\psi=B$. As for conserved quantities, total energy and cross-helicity has analogous expressions, but magnetic helicity is replaced by the \emph{mean-square magnetic potential}, defined as
$$
\int_{\T^2}  \abs{\psi}^2 dx.
$$
Mean-square magnetic potential is conserved by smooth solutions, and in 
\cite{CKS} the conservation was shown for weak solutions $(u,B)$ with the regularity $u\in C([0,T]; B^{\alpha_1}_{3,\infty})$ and $B\in C([0,T]; B^{\alpha_2}_{3,\infty})$ for $\alpha_1 + 2 \alpha_2 > 1$. 

Next, observe that \eqref{MHD3} implies that $u\cdot B^\perp$ is a div-curl product. Consequently, if we have a sequence of velocity and magnetic fields $(u_j,B_j)\rightharpoonup (u,B)$ converging weakly in $L^2$ and such that $\{\nabla\cdot u_j\}$ and $\{\nabla\cdot E_j\}$ are in a compact subset of $H^{-1}$, then $u_j\cdot B_j^\perp\overset{*}{\rightharpoonup} u\cdot B^\perp$ in the space of measures. In other words \eqref{MHD2-2D} is stable under weak convergence in $L^2$. Another way of writing \eqref{MHD2-2D} is by using the stream functions of $u$ and $B$. Indeed, if we write $u=\nabla^\perp\phi$ and $B=\nabla^\perp\psi$, with $\langle\phi\rangle=\langle\psi\rangle=0$, then \eqref{MHD2-2D} becomes
\begin{equation} \label{Time evolution of psi in 2D}
\partial_t \psi + J(\phi,\psi) = 0,
\end{equation}
where we write, as usual, $J(\phi,\psi)=\nabla\phi\cdot\nabla^\perp\psi$ for the Jacobian determinant of the mapping $(\phi,\psi) \colon \T^2\to\R^2$. Observe that the same equation appears also for the 2D Euler equations, where we replace $\psi$ by the vorticity $\omega=\partial_1u_2-\partial_2u_1$ and $\phi$ by the velocity potential $v=\nabla^\perp\phi$. However, here we do not assume any coupling between $\phi$ and $\psi$, and treat \eqref{Time evolution of psi in 2D} as a passive scalar equation. 

The form \eqref{Time evolution of psi in 2D} allows us to prove conservation of the mean-square magnetic potential under very mild conditions:

\begin{thm} \label{Helicity preservation theorem}
Suppose that $(u,B) \in C_w([0,T[;L^2(\T^2))$ is a weak solution of \eqref{MHD3} and \eqref{MHD2-2D}. Then the mean-square magnetic potential is conserved.
\end{thm}

We point out that the analogous result for the 2D Euler equations, namely the conservation of enstrophy $\tfrac{1}{2}\int|\omega|^2dx$ is well-known \cite{Eyi2,LMN}, and the proof is based on the theory of renormalised solutions. Here we give an alternative, short proof, again emphasising that compensated compactness lies at the heart of the matter. For the proof we first recall the $\mathcal{H}^1$ regularity theory of Coifman, Lions, Meyer and Semmes from~\cite{CLMS}, more precisely the following adaptation of the classical Wente inequality to the torus $\T^2$ (see~\cite[Theorem A.1]{FMS}):
\begin{lem} \label{Jacobian estimate}
When $(f_1,f_2,f_3) \in W^{1,2}(\T^2,\R^3)$, we have
\begin{align} \label{Jacobian inequality}
\int_{\T^2} f_1(x) J_{(f_2,f_3)}(x) \, dx
&\lesssim \norm{f_1}_{\operatorname{BMO}(\T^2)} \norm{J_{(f_2,f_3)}}_{\mathcal{H}^1(\T^2)} \notag \\
&\lesssim \norm{\nabla f_1}_{L^2(\T^2)} \norm{\nabla f_2}_{L^2(\T^2)} \norm{\nabla f_3}_{L^2(\T^2)}.
\end{align}
\end{lem}

The left-hand side of \eqref{Jacobian inequality} can be understood in terms of $\mathcal{H}^1$--$\text{BMO}$ duality, but we in fact only require \eqref{Jacobian inequality} where the left-hand side is Lebesgue integrable.
 
\begin{proof}[Proof of Theorem \ref{Helicity preservation theorem}]
First let us assume that $u$ and $B$ are smooth. Then, using \eqref{Time evolution of psi in 2D}, we obtain after integration by parts 	
\begin{equation}\label{e:2Dcomputation}
\frac{d}{dt} \frac{1}{2} \int_{\T^2} |\psi|^2 dx = \int_{\mathbb{T}^2} \psi J(\psi,\phi) \, dx = -\int_{\mathbb{T}^2} \phi J(\psi,\psi) \, dx = 0.
\end{equation}
For the general case note first that under the assumption on $(u,B)$, using the compact embedding $W^{1,2}\hookrightarrow L^2$ the stream functions $\phi,\psi$ belong to $C([0,T];L^2(\T^2))$ with
$\nabla\phi,\nabla\psi\in C_w([0,T];L^2(\T^2))$. Then the computation \eqref{e:2Dcomputation} can be carried out using standard regularisation of $\psi,\phi$ and the uniform bound in \eqref{Jacobian inequality}.  
\end{proof}
 
Theorem \ref{Helicity preservation theorem} implies the following corollary:

\begin{cor} \label{MHD 2D corollary}
Suppose $(u,B) \in C_w([0,T[;L^2(\T^2))$ is a weak solution of \eqref{MHD3} and \eqref{MHD2-2D}. Then either $B \equiv 0$ or there exists a constant $c > 0$ such that $\int_{\T^2} \abs{B}^2 dx \ge c$ for every $t \in [0,T[$.
\end{cor}

\begin{proof}
The proof follows by using the Poincar\'{e} inequality at every $t \in [0,T[$ to estimate $\int_{\T^2} \abs{B(x,t)}^2 dx = \int_{\T^2} \abs{\nabla \psi(x,t)}^2 dx \ge C \int_{\T^2} \abs{\psi(x,t)}^2 dx$.
\end{proof}

Thus, in 2D even if the kinetic and magnetic energies $\int_{\T^2} \abs{u}^2 dx$ and $\int_{\T^2} \abs{B}^2 dx$ may fluctuate (and indeed, numerical experiments indicate anomalous dissipation of the total energy even in 2D \cite{BiskampWelter}), by Corollary \ref{MHD 2D corollary} it is impossible for the magnetic energy to dissipate to zero. 

Finally, we remark that although it is natural to ask whether an analogue of Theorem \ref{Helicity preservation theorem} holds in the whole space $\R^2$, in fact square integrable divergence-free vector fields do not in general have a square integrable stream function in $\R^2$. This is shown in Appendix \ref{The ill-definedness of magnetic helicity and mean-square magnetic potential in the whole space}.

\section{Plane-wave analysis}

Recall that the ideal MHD system in 3D can be written for a \emph{state variable} $(u,S,B,E)$ in terms of the conservation laws \eqref{Linearised MHD1}-\eqref{Linearised MHD3} with the constitutive set $K$, defined in \eqref{e:K}. 
The framework introduced by Tartar amounts to an analysis of one-dimensional oscillations compatible with \eqref{Linearised MHD1}-\eqref{Linearised MHD3} -- the wave-cone -- and then the interaction of the wave-cone with the constitutive set. We carry out this analysis in this section.

\subsection{The wave cone and the lamination convex hull} \label{The wave cone and the lamination convex hull}
\emph{Plane waves} are one-dimensional oscillations of the form $(x,t) \mapsto h((x,t) \cdot \xi) V$ with
\[V = (u,S,B,E) \in \R^{15},\]
$\xi = (\xi_x,\xi_t) \in (\R^3 \times \R) \setminus \{0\}$ and $h \colon \R \to \R$. For a plane wave solution, \eqref{Linearised MHD1}--\eqref{Linearised MHD3} become
\begin{align}
& \xi_x \cdot u = \xi_x \cdot B = 0, \label{Wave cone condition 1} \\
& \xi_t u + S \xi_x = 0, \label{Wave cone condition 2} \\
& \xi_t B + \xi_x \times E = 0 . \label{Wave cone condition 3}
\end{align}
In the following, we will write, with a slight abuse of notation, \eqref{Wave cone condition 1}--\eqref{Wave cone condition 3} in the concise form $V \xi = 0$.

\begin{defin}
The \emph{wave cone} for ideal MHD is
\[\Lambda_0 = \{V = (u,S,B,E) \in \R^{15} \colon \exists \xi \in \R^{4} \setminus \{0\} \text{ such that \eqref{Wave cone condition 1}--\eqref{Wave cone condition 3} hold}\}.\]
We also denote
\[\Lambda = \{V = (u,S,B,E) \in \R^{15} \colon \exists \xi \in (\R^3 \setminus \{0\}) \times \R \text{ such that \eqref{Wave cone condition 1}--\eqref{Wave cone condition 3} hold}\}.\]
If $V_1,V_2 \in \R^{15}$ satisfy $V_1 - V_2 \in \Lambda$, then $[V_1,V_2] \subset \R^{15}$ is called a $\Lambda$-\emph{segment}.
\end{defin}

In the convex integration process we will use $\Lambda$ instead of $\Lambda_0$, as the requirement $\xi_x \neq 0$ is crucial to many of the arguments. We next define lamination convex and $\Lambda$-convex hulls.

Given a set $Y \subset \R^{15}$ we denote $Y^{0,\Lambda} \defeq Y$ and define inductively
\[Y^{N+1,\Lambda} \defeq Y^{N,\Lambda} \cup \{\lambda V + (1-\lambda) W \colon \lambda \in [0,1], \; V, W \in Y^{N,\Lambda}, \; V-W \in \Lambda\}\]
for all $N \in \N_0$.

\begin{defin}
When $Y \subset \R^{15}$, the \emph{lamination convex hull} of $Y$ (with respect to $\Lambda$) is
\[Y^{lc,\Lambda} \defeq \bigcup_{N \ge 0} Y^{N,\Lambda}.\]
\end{defin}

It is well-known that semiconvex hulls can be expressed by duality in terms of measures, see e.g.~\cite{Pedregal},~\cite{Kir} and~\cite{CS}.

\begin{defin}
Let $Y \subset \R^{15}$. The set of \emph{laminates of finite order} (with respect to $\Lambda$), denoted $\mathcal{L}(Y)$, is the smallest class of atomic probability measures supported on $Y$ with the following properties:
\begin{enumerate}[\upshape (i)]
\item $\mathcal{L}(Y)$ contains all the Dirac masses with support in $Y$.

\item $\mathcal{L}(Y)$ is closed under splitting along $\Lambda$-segments inside $Y$.
\end{enumerate}
Condition (ii) means that if $\nu = \sum_{i=1}^M \nu_i \delta_{V_i} \in \mathcal{L}(Y)$ and $V_M \in [Z_1,Z_2] \subset Y$ with $Z_1-Z_2 \in \Lambda$, then
\[\sum_{i=1}^{M-1} \nu_i \delta_{V_i} + \nu_M (\lambda \delta_{Z_1} + (1-\lambda) \delta_{Z_2}) \in \mathcal{L}(Y),\]
where $\lambda \in [0,1]$ such that $V_M = \lambda Z_1 + (1-\lambda) Z_2$.
\end{defin}

\begin{rem} \label{Remark on lamination convex hull}
Given $V \in Y^{N,\Lambda}$, we may write $V = \lambda_1 V_1 + \lambda_2 V_2$, where
\[V_1,V_2 \in Y^{N-1,\Lambda}, \qquad 0 \le \lambda_1 \le 1, \qquad \lambda_1 + \lambda_2 = 1, \qquad V_1-V_2 \in \Lambda.\]
Similarly, we write $V_1 = \lambda_{1,1} V_{1,1} + \lambda_{1,2} V_{1,2}$. Repeating this process, by induction we arrive at a finite-order laminate with support in $Y$ and barycentre $V$:
\[\nu = \sum_{\mathbf{j} \in \{1,2\}^N} \mu_\mathbf{j} \delta_{V_\mathbf{j}}, \qquad \text{supp}(\nu) \subset Y, \qquad \bar{\nu} = V,\]
where $\mu_\mathbf{j} = \mu_{j_1,\ldots,j_N} = \lambda_{j_1} \ldots \lambda_{j_1,\ldots,j_N}\in [0,1]$.
\end{rem}

In addition to the lamination convex hull, another, potentially larger, hull is used in convex integration theory. In order to define it we recall the notion of $\Lambda$-convex functions.

\begin{defin}
A function $f \colon \R^{15} \to \R$ is said to be \emph{$\Lambda$-convex} if the function $t \mapsto f(V + tW) \colon \R \to \R$ is convex for every $V \in \R^{15}$ and every $W \in \Lambda$.
\end{defin}

While the lamination convex hull is defined by taking convex combinations, the $\Lambda$-convex hull $Y^\Lambda$ of $Y \subset \R^{15}$ is defined as the set of points that cannot be separated from $Y$ by $\Lambda$-convex functions.

\begin{defin}
When $Y \subset \R^{15}$ is compact, the $\Lambda$-\emph{convex hull} $Y^\Lambda$ consists of points $W \in \R^{15}$ with the following property: if $f \colon \R^{15} \to \R$ is $\Lambda$-convex and $f|_Y \le 0$, then $f(W) \le 0$.
\end{defin}

\subsection{Normalisations of the constitutive set}
In order to produce bounded solutions of 3D MHD we consider normalised versions of the constitutive set $K$. We wish to prescribe both the total energy density $(\abs{u}^2+\abs{B}^2)/2$ and the cross helicity density $u \cdot B$, but for this aim it is obviously not enough to prescribe $\abs{u}$ and $\abs{B}$. However, by using the \emph{Els\"{a}sser variables}
\[z^\pm \defeq u \pm B\]
we can write $(\abs{u}^2 + \abs{B}^2)/2 = (\abs{z^+}^2 + \abs{z^-}^2)/4$ and $u \cdot B = (\abs{z^+}^2-\abs{z^-}^2)/4$, and thus it suffices to prescribe $\abs{z^\pm}$. This motivates the normalisation given below; recall that $K \defeq \{(u,S,B,E) \colon S = u \otimes u - B \otimes B + \Pi I, \; \Pi \in \R, \; E = B \times u\}$.

\begin{defin}\label{d:Krs}
Whenever $r,s > 0$, we denote
\begin{equation}\label{e:Krs}
K_{r,s} \defeq \{(u,S,B,E) \in K \colon \abs{u+B} = r, \; \abs{u-B} = s, \; \abs{\Pi} \le rs\}.
\end{equation}
\end{defin}

As pointed out in Section \ref{s:system}, the Maxwell system is essentially closed under weak convergence; the scalar product $B\cdot E$ is weakly continuous. As an immediate consequence the $\Lambda$-convex hull $K_{r,s}^\Lambda$ has empty interior (in 3D, and also in 2D). Indeed, Tartar's result in Example 4 \cite{Tartar} is based on the fact that the quadratic expression $Q(u,S,B,E) \defeq B \cdot E$ satisfies 
$$
Q(V)=0\quad\textrm{ for all }V\in \Lambda_0,
$$
and consequently, $Q$ is $\Lambda_0$-affine. Then we deduce that 
$$
K_{r,s}^{\Lambda_0}\subset \mathscr{M},
$$
where $\mathscr{M}$ is the set in \eqref{e:M}.

\bigskip

In 3D, assume now $B \times u \neq 0$. Then \eqref{Wave cone condition 1} implies, up to normalisation, that $\xi_x = B \times u$, and then \eqref{Wave cone condition 3} yields $\xi_t B = - (B \times u) \times E = (E \cdot u) B - (E \cdot B) u = (E \cdot u) B$. Thus, whenever $B \times u \neq 0$, \eqref{Wave cone condition 1}--\eqref{Wave cone condition 3} reduce to the conditions
\begin{equation} \label{Wave cone condition 4}
S (B \times u) + (E \cdot u) u = 0, \qquad B \cdot E = 0
\end{equation}
which is an easier condition to check in the sequel.





\section{Discussion of the convex integration scheme in 3D} \label{Discussion of the convex integration scheme in 3D}
The standard way of finding nontrivial compactly supported solutions for equations of fluid dynamics was first presented in~\cite{DLS09} and axiomatised in~\cite{Szekelyhidi}. We describe it briefly in the case of Theorem \ref{MHD 3D theorem}.

With a bounded domain $\Omega \subset \R^4$ fixed, it suffices to find a solution $V$ of the relaxed MHD equations $\mathcal{L}(V) = 0$ such that $V(x,t) \in K_{2,1}$ a.e. in $\Omega$ ($K_{2,1}$ defined in \eqref{e:Krs}) and $V(x,t) = 0$ a.e. outside $\Omega$. One intends to construct $V$ as a limit of subsolutions, that is, mappings $V_\ell$ solving $\mathcal{L}(V_\ell) = 0$ and taking values in $K^{lc,\Lambda}_{2,1}$.

The basic building blocks of the construction are plane waves which oscillate in directions of $\Lambda$. In order to prevent harmful interference of the waves and to make the eventual solutions compactly supported, one needs to localise the plane waves. The localisation is customarily carried out by constructing potentials. This causes small error terms, and in order for each $V_\ell$ to take values in the lamination convex hull, one hopes to prove that the hull has non-empty interior. The specifics of the convex integration scheme vary (see e.g.~\cite{DLS09},~\cite{CFG} and~\cite{CS} for three different approaches in fluid dynamics and~\cite{Szekelyhidi,Kir} for a more general discussion).

In the case of 3D MHD, the process is more subtle, as $K^{lc,\Lambda}_{2,1}$ has an empty interior, more precisely $K^{lc,\Lambda}_{2,1}\subset\mathscr{M}$. Therefore, although we may proceed with the 'symmetric (fluid) part' $u$ and $S$, the 'anti-symmetric (electromagnetic) part' $B$ and $E$ needs special attention.

As a first step towards overcoming the emptiness of $\text{int}(K^{lc,\Lambda}_{2,1})$, we construct a pair of non-linear potential operators $P_B$ and $P_E$ that satisfy $\nabla \cdot P_B[\varphi,\psi]= 0$, $\partial_t P_B[\varphi,\psi) + \nabla \times P_E[\varphi,\psi] = 0$ and $P_E[\varphi,\psi] \cdot P_E[\varphi,\psi] = 0$ for all $\varphi,\psi \in C^\infty(\R^4)$. (For $u$ and $S$ we simply use the potentials in \cite{DLS09} for the Euler equations.) We add the localised plane waves \emph{within} $P_B$ and $P_E$; despite their non-linearity, $P_B$ and $P_E$ have cancellation properties which allow them to map suitable sums of localised plane waves to sums of localised plane waves (up to a small error term).


As a drawback, $P_B$ and $P_E$ do not allow oscillating plane waves for every $\Lambda$-segment -- their applicability depends not only on the direction but also on the location of the segment. We consider $\Lambda$-segments for which $P_B$ and $P_E$ give plane waves and call them \emph{good $\Lambda$-segments} or $\Lambda_g$-segments. This leads us to study $K_{2,1}^{lc,\Lambda_g}$, the restricted lamination convex hull of $K_{2,1}$ in terms of $\Lambda_g$.

A priori, $\Lambda_g$ is a rather large subset of $\Lambda$-segments. However, even though $K_{2,1}^{lc,\Lambda}$ has non-empty interior relative to $\mathscr{M}$, the electromagnetic part of $K_{2,1}^{lc,\Lambda_g}$ is rigid: the constraint $E = B \times u$ holds for all $(u,S,B,E) \in K^{lc,\Lambda}_{2,1}$. Nevertheless, as the in-approximation formulation of convex integration shows, the iterative step happens at relatively open sets and it is a limit procedure which leads to the inclusion in closed sets. Thus, in this case the size of $\Lambda_g$ saves the day; as it turns out, for \emph{relatively open} subsets $U \subset \mathscr{M}$ we have $U^{lc,\Lambda} = U^{lc,\Lambda_g}$. This eventually allows us to apply the Baire category framework of convex integration in $\U_{2,1} \defeq \text{int}_{\mathscr{M}}(K^{lc,\Lambda}_{2,1})$. We present useful characterisations of $\U_{2,1}$ in Theorem \ref{Hull theorem}; in particular, $\operatorname{int}_{\mathscr{M}}(K^{lc,\Lambda}_{2,1}) = \cup_{0 \le \tau < 1} K_{2\tau,\tau}^{lc,\Lambda}$. Theorem \ref{Hull theorem} is the most technically difficult part of the paper and the heart of the convex integration scheme. The proof of Theorem \ref{MHD 3D theorem} is then completed in \textsection \ref{The proof of MHD 3D theorem} .

Notice that actually, we do not compute the exact hull $K^{lc,\Lambda}_{2,1}$. However, the formula $\operatorname{int}_{\mathscr{M}}(K^{lc,\Lambda}_{2,1}) = \cup_{0 \le \tau < 1} K_{2\tau,\tau}^{lc,\Lambda}$ turns out to give us enough information about $K^{lc,\Lambda}_{2,1}$. The formula is used in a similar manner as in~\cite{CS}.

\section{Potentials in 3D} \label{Potentials in 3D}
We wish to find potentials corresponding to $\Lambda$-segments. For the fluid variables $(u,S)$, we simply use the potentials of~\cite{DLS09,DLS10} for the Euler equations. In the case of the electromagnetic variables $(B,E)$, the question about existence of potentials is more subtle because of the non-linear constraint $B \cdot E = 0$ that the potentials need to obey. This issue is studied in \textsection \ref{Maxwell two-forms}--\ref{Potentials for good Lambda-segments}.

%
%
%

\subsection{Potentials for the fluid side} \label{Potentials for the symmetric side}
We recall from~\cite{DLS09,DLS10} that potentials for the fluid part, i.e., the variables $u$ and $S$, can be obtained as follows. First of all, recall that
\eqref{Linearised MHD1}-\eqref{Linearised MHD2} can be written equivalently for the symmetric $4\times 4$ matrix
\begin{equation}\label{e:U}
U=\begin{pmatrix}S & u \\u^T & 0\end{pmatrix}
\end{equation}
as $\nabla_{x,t}\cdot U=0$. With this notation \eqref{Wave cone condition 1}-\eqref{Wave cone condition 2} (i.e. belonging to the wave-cone) is equivalent to $U\xi=0$ for some $\xi\in\R^4\setminus\{0\}$. Let us denote $\R^{4 \times 4}_{\operatorname{sym},0} \defeq \{U \in \R^{4 \times 4}_{\operatorname{sym}} \colon U_{4,4} = 0\}$.
\begin{lem}\label{l:fluidpotential}
Suppose $U\in \R^{4 \times 4}_{\operatorname{sym},0}$ such that $U \xi = 0$ for some $(\xi_x,\xi_t) \in (\R^3 \setminus \{0\}) \times \R$. Then there exists $P_U \colon C^\infty(\R^3 \times \R) \to C^\infty(\R^3 \times \R;\R^{4 \times 4}_{\operatorname{sym},0})$ with the following properties:
\begin{enumerate}[\upshape (i)]
\item $\nabla \cdot P_U[\phi] = 0$ for every $\phi \in C^\infty(\R^3 \times \R)$.
\item If $\phi(x,t)=h((x,t) \cdot \xi)$ for some $h\in C^{\infty}(\R^3\times\R)$, then we have $P_U[\phi](x,t) = h^{\prime \prime}((x,t) \cdot \xi) U$ for all $(x,t) \in \R^3 \times \R$.
\end{enumerate}
\end{lem}

This lemma essentially follows from the proof of \cite[Proposition 3.2]{DLS09}. For the convenience of the reader we sketch a simplified proof, following the exposition in \cite{Szekelyhidi1}:
\begin{proof}
 As noted in \cite{DLS10,Szekelyhidi1}, a matrix-valued quadratic homogeneous polynomial $P:\R^4\to\R^{4\times 4}$ gives rise to a differential operator $P(\partial)$ as required in the lemma, if $P=P(\eta)$ satisfies 
 $$
 P\eta=0,\,P^T=P,\,Pe_4=0,\,P(\xi)=U.
 $$
 Elementary examples satisfying the first 3 conditions above are given by $P(\eta)=\tfrac{1}{2}(R\eta\otimes Q\eta+Q\eta\otimes R\eta)$ for antisymmetric $4\times 4$ matrices $R,Q$ such that $Re_4=0$. In particular, for any $a,b\perp\xi$ with $a\perp e_4$, set $R=a\otimes \xi-\xi\otimes a$ and $Q=b\otimes \xi-\xi\otimes b$, to obtain $P_{a,b}(\eta)$. One quickly verifies that $P_{a,b}(\xi)=\tfrac{1}{2}(a\otimes b+b\otimes a)$. Since any $U\in \R^{4\times 4}_{sym,0}$ with $U\xi=0$ can be written as a linear combination 
 $$
 U=\sum_i\tfrac{1}{2}(a_i\otimes b_i+b_i\otimes a_i)
 $$
 for vectors $a_i,b_i\in\R^4$ with $a_i\cdot\xi=b_i\cdot \xi=0$ and $a_i\cdot e_4=0$, we obtain $P_U$ as required in the lemma as
 $$
 P_U(\eta)=\sum_iP_{a_i,b_i}(\eta).
 $$
 \end{proof}

\subsection{Wave cone conditions on $u$, $B$ and $E$}
It will turn out that when we choose which $\Lambda$-directions to use, we have much more freedom in the choice of $S$ than the three other variables $u$, $B$ and $E$. Recall that in 3D, the wave cone conditions are
\begin{align}
& \xi_x \cdot u = \xi_x \cdot B = 0, \label{Wave cone condition 1b} \\
& \xi_t u + S \xi_x = 0, \label{Wave cone condition 2b} \\
& \xi_t B + \xi_x \times E = 0. \label{Wave cone condition 3b}
\end{align}
We can typically first find $u,B,E,\xi$ satisfying \eqref{Wave cone condition 1b} and \eqref{Wave cone condition 3b} and afterwards choose $S$ satisfying \eqref{Wave cone condition 2b}. This motivates the following observation.

\begin{lem} \label{Wave cone in u, b and a}
Let $u,B,E \in \R^3$. The following conditions are equivalent.
\begin{enumerate}[\upshape (i)]
\item \eqref{Wave cone condition 1b} and \eqref{Wave cone condition 3b} have a solution $\xi \in (\R^3 \setminus \{0\}) \times \R$.
\item $B \cdot E = 0$.
\end{enumerate}
\end{lem}

\begin{proof}
We first show that (i) $\Rightarrow$ (ii). Choose a solution $\xi \in (\R^3 \setminus \{0\}) \times \R$ of \eqref{Wave cone condition 1b} and \eqref{Wave cone condition 3b}. If $\xi_t \neq 0$, then \eqref{Wave cone condition 3b} gives $B \cdot E = - (\xi_x \times E \cdot E)/\xi_t = 0$. If $\xi_t = 0$, then \eqref{Wave cone condition 3b} gives $E = k \xi_x$ for some $k \in \R$, so that \eqref{Wave cone condition 1b} gives $B \cdot E = 0$.

We then show that (ii) $\Rightarrow$ (i). If $B \times u \neq 0$, we choose $\xi_x = B \times u$. Since $B \cdot E = 0$, we may write $E = c_1 \, B \times u + c_2 B \times (B \times u)$ for some $c_1,c_2 \in \R$. (The set $\{B,B \times u, B \times (B \times u)\}$ is an orthogonal basis of $\R^3$.) Thus $\xi_x \times E = c_2 \abs{B \times u}^2 B$ and we may choose $\xi_t = - c_2 \abs{B \times u}^2$.  If, on the other hand, $B \times u = 0$, we may set $\xi_t = 0$ and choose $\xi_x = a$ if $E \neq 0$ and any $\xi_x \in \{B\}^\perp \setminus \{0\}$ if $E = 0$.
\end{proof}

\subsection{Maxwell two-forms} \label{Maxwell two-forms}
Our aim in the rest of this chapter is to find potentials for the variables $B$ and $E$. We carry out this task using the formalisms of two-forms and bivectors in $\mathbb{R}^4$. 
In electromagnetics, it is customary to express $(B,E) \in \mathbb{R}^3 \times \mathbb{R}^3$  as a unique bivector $\omega \in \Lambda^2(\mathbb{R}^4)$ via the identification
\begin{equation} \label{Definition of F}
\begin{aligned}
\omega & \defeq B_1 dx^2 \wedge dx^3 + B_2 dx^3 \wedge dx^1 + B_3 dx^1 \wedge dx^2 \\
&+ E_1 dx^1 \wedge dx^4 + E_2 dx^2 \wedge dx^4 + E_3 dx^3 \wedge dx^4
\end{aligned}
\end{equation}
(see~\cite{Des}). We write $\omega \cong (B,E)$. Then, Gauss' law and Maxwell-Faraday law are written concisely via differential forms:
\begin{equation} \label{Simplified form of relaxed MHD on b and a}
\nabla \cdot B = 0 \quad \text{and} \quad \partial_t B + \nabla \times E = 0 \qquad \Longleftrightarrow \qquad d \omega = 0,
\end{equation}
i.e., $\omega$ is an exact two-form called {\it Maxwell two-form} or {\it electromagnetic two-form}.

Recall that in addition to \eqref{Simplified form of relaxed MHD on b and a}, we also need $E$ and $B$ to satisfy $B \cdot E = 0$. We express the latter condition in the language of bivectors: 
\begin{align*}
B \cdot E = 0 \qquad
&\Longleftrightarrow \qquad \omega \wedge \omega = 2 B \cdot E \, dx^1 \wedge dx^2 \wedge dx^3 \wedge dx^4 = 0 \\
&\Longleftrightarrow \qquad \omega = v \wedge w \quad \text{for some } v,w \in \R^4
\end{align*}
(where the last equality showing that $\omega$ is simple will be proved in the forthcoming Proposition \ref{Proposition on characterisations of simple forms}). Here and in the sequel, we identify a vector $v \in \R^4$ and a 1-form $\sum_{i=1}^4 v_i dx^i$.

Our nonlinear constraint simplifies to
\begin{equation}\label{Mbivectors}
\mathscr{M}=\{(u,S,\omega):  \omega \wedge \omega=0\} 
\end{equation}
and the  wave cone conditions for $(B,E)$, \eqref{Wave cone condition 1b} and \eqref{Wave cone condition 3b}, are reduced to 
\begin{equation}\label{waveconebivectors}
\omega \wedge \xi=0.
\end{equation}
If such a $\xi$ is found, in view of Lemma \ref{Wave cone in u, b and a} it can be modified to verify  $\xi_x \cdot u=0$ as well. Thus  it only remains to verify the condition involving $S$, i.e., \eqref{Wave cone condition 2b}.

It turns out that the interaction of \eqref{Mbivectors} and \eqref{waveconebivectors} is very neat with the forms formalism. This is the content of the next section.

\subsection{$\Lambda$-segments in terms of simple bivectors} \label{Lambda-segments in terms of simple two-forms}
The following well-known proposition collects characterisations equivalent to the condition $B \cdot E = 0$ (The Pl\"{u}cker identity for the bivector $\omega$).

\begin{prop} \label{Proposition on characterisations of simple forms}
Let $\omega \cong (B,E) \in \R^3 \times \R^3$. The following conditions are equivalent:
\begin{enumerate}[\upshape (i)]
\item $\omega$ is \emph{degenerate}, that is, $\omega \wedge \omega = 0$.

\item $\omega$ is \emph{simple}, that is, $\omega = v \wedge w$ for some $v,w \in \R^4$, called the \emph{factors of} $\omega$.

\item $B \cdot E = 0$ 
\item $\omega \wedge \xi=0$ for some $\xi \in (\R^3 \setminus \{0\}) \times \R$.
\end{enumerate}
\end{prop}

\begin{proof}
The equivalence of (i) and (iii) was already noted, and (ii) clearly implies (i). Suppose then (iii) holds; our aim is to prove (ii). If $E = 0$, choose any $v_x,w_x \in \R^3$ such that $v_x \times w_x = b$. Then $(v_x,0) \wedge (w_x,0) \cong (v_x \times w_x, 0) = (B,0)$. If $E \neq 0$, then $(E,0) \wedge (B \times E/\abs{E}^2,1) \cong (B,E)$, giving (ii).

The implication (iii) $\Rightarrow$ (iv) follows from Lemma \ref{Wave cone in u, b and a}, and the proof of Lemma \ref{Wave cone in u, b and a} also gives (iv) $\Rightarrow$ (iii). Alternatively, (iv) $\Rightarrow$ (ii) follows from Proposition \ref{Equivalent conditions for wave cone condition} below.
\end{proof}

Using Proposition \ref{Proposition on characterisations of simple forms}, we formulate some useful further characterisations of \eqref{waveconebivectors}.

\begin{prop} \label{Equivalent conditions for wave cone condition}
Suppose $\omega = v \wedge w \neq 0$ and $\xi \in (\R^3 \setminus \{0\}) \times \R$. The following conditions are equivalent:
\begin{enumerate}[\upshape (i)]
\item $\omega \wedge \xi = 0$.

\item $\xi \in \spann\{v,w\}$.

\item $\omega = \tilde{v} \wedge \xi$ for some $\tilde{v} \in \spann\{v,w\} \setminus \{0\}$.
\end{enumerate}
\end{prop}

\begin{proof}
For (i) $\Rightarrow$ (ii) suppose $v \wedge w \wedge \xi = 0$. We may thus write $c_1 v + c_2 w + c_3 \xi = 0$, where $\{c_1,c_2,c_3\} \neq \{0\}$. If $c_3 = 0$, we get a contradiction with $v \wedge w \neq 0$, and therefore $\xi \in \spann\{v,w\}$. For (ii) $\Rightarrow$ (iii) choose $\tilde{v} \in \spann\{v,w\} \setminus \{0\}$ with $\tilde{v} \cdot \xi = 0$. After normalising $\tilde{v}$ we get $\tilde{v} \wedge \xi = v \wedge w$. The direction (iii) $\Rightarrow$ (i) is clear.
\end{proof}

Recall that every $\Lambda$-segment is contained in $\mathscr{M}$. We give equivalent characterisations for this condition.

\begin{prop} \label{Characterisation of Lambda segments}
Suppose that $\omega_0$ and $\omega \neq 0$ are simple bivectors and that $\omega \wedge \xi = 0$, where $\xi \in (\R^3 \setminus \{0\}) \times \R$. The following conditions are equivalent:
\begin{enumerate}[\upshape (i)]
\item $\omega_0 + t \omega$ is simple for all $t \in \R$.

\item $\omega_0 \wedge \omega = 0$.

\item We can write $\omega = v \wedge \xi$ and either $\omega_0 = v_0 \wedge \xi$ or $\omega_0 = v \wedge w_0$.
\end{enumerate}
\end{prop}

\begin{proof}
The equivalence (i) $\Leftrightarrow$ (ii) is clear since $\omega_0 \wedge \omega = \omega \wedge \omega_0$. The direction (iii) $\Rightarrow$ (ii) is also clear, and we complete the proof by showing that (ii) $\Rightarrow$ (iii). The case $\omega_0 = 0$ being clear, we assume that $\omega_0 \neq 0$.

Use Proposition \ref{Equivalent conditions for wave cone condition} to write $\omega = \tilde{v} \wedge \xi$ for some $\tilde{v} \in \R^4 \setminus \{0\}$. Also write $\omega_0 = \tilde{v}_0 \wedge \tilde{w}_0$. Since $\omega_0 \wedge \omega = 0$ and $\omega_0 \neq 0$ by assumption, we conclude that $\tilde{v}_0 = d_1 \tilde{w}_0 + d_2 \tilde{v} + d_3 \xi$ for some $d_1,d_2,d_3 \in \R$.

If $d_3 = 0$, we set $v = \tilde{v}$ and $w_0 = d_2 \tilde{w}_0$. Next, if $d_3 \neq 0$ and $d_2 \neq 0$, we choose $v = \tilde{v} + (d_3/d_2) \xi$ and $w_0 = d_2 \tilde{w}_0$. Finally, if $d_3 \neq 0$ and $d_2 = 0$, we select $v = \tilde{v}$ and $v_0 = - d_3 \tilde{w}_0$.
\end{proof}

\subsection{Clebsch variables}

Now \eqref{Simplified form of relaxed MHD on b and a} means that $\omega$ is closed and thus, by Poincar\'{e} lemma, exact: $\omega = d \alpha$. Here the so-called \emph{electromagnetic four-potential} $\alpha$ is of course not unique. We specify a choice of $\alpha$ below. Recall from \eqref{Mbivectors} that our potential $\alpha$ is required to satisfy
\[
d\alpha \wedge d\alpha=0.
\]
This fact, among other things, motivates us to set $\alpha = \varphi \, d \psi$ which leads to $\omega = d \alpha = d \varphi \wedge d\psi$; here $\phi,\psi \in C^\infty(\R^4)$ are traditionally called \emph{Clebsch variables} or \emph{Euler potentials}.

\begin{defin}
We define $P_B, P_E \colon C^\infty(\R^4) \times C^\infty(\R^4) \to C^\infty(\R^4;\R^3)$ via
\begin{equation} \label{Potentials for b and a}
d\varphi \wedge d\psi \cong (\nabla \varphi \times \nabla \psi, \partial_t \psi \nabla \varphi - \partial_t \varphi \nabla \psi) \eqdef (P_B[\varphi,\psi],P_E[\varphi,\psi]).
\end{equation}
\end{defin}

With the Clebsch variables at our disposal we make a natural Ansatz on the electromagnetic side of the localised plane waves. 
Fix $V_0=(u_0,S_0,\omega_0) \in \mathscr{M}$, $V=(u,S,\omega)\in \Lambda$ with $\xi \in (\R^3 \setminus \{0\}) \times \R$ being a solution to \eqref{Wave cone condition 1}--\eqref{Wave cone condition 3} and $\omega_0 \wedge \omega=0$.

 Use the simplicity of $\omega_0$ to  write $\omega_0 = v_0 \wedge w_0$, and recall the operator $P_{U}$ given by Lemma \ref{l:fluidpotential}, with $U$ given in \eqref{e:U}. 

Fix a cube $Q \subset \R^4$ and a cutoff function $\chi \in C_c^\infty(Q)$. Given $h \in C^\infty(\R)$ and $\ell \in \N$, our aim is to find $\phi_\ell$, $\varphi_\ell$ and $\psi_\ell$ such that
\begin{equation} \label{Aim on potentials}
V_\ell \defeq ((u_0,S_0) + P_{U}(\phi_\ell), d \varphi_\ell \wedge d \psi_\ell) = V_0 + \chi(x,t) h^{\prime \prime}(\ell (x,t) \cdot \xi) V + O \left( \frac{1}{\ell} \right)
\end{equation}
and $V_\ell \rightharpoonup V_0$ in $L^2(Q;\R^{15})$. The choice of $\phi_\ell$ is specified in Lemma \ref{Potentials for the symmetric side}. For the electromagnetic part we define Clebsch variables of the form
\begin{align}
& \varphi_\ell(x,t) \defeq v_0 \cdot (x,t) + \frac{c_1 \chi(x,t) h'(\ell(x,t) \cdot \xi)}{\ell}, \label{Potential Ansatz 1} \\
& \psi_\ell(x,t) \defeq w_0 \cdot (x,t) + \frac{c_2 \chi(x,t) h'((x,t) \cdot \xi)}{\ell}. \label{Potential Ansatz 2}
\end{align}
In \eqref{Potential Ansatz 1}--\eqref{Potential Ansatz 2}, we use $h'$ instead of $h$ in order to be consistent with the scaling on the fluid part.

The Ansatz \eqref{Potential Ansatz 1}--\eqref{Potential Ansatz 2} yields
\begin{equation} \label{Cancellation formula}
d \varphi_\ell(x,t) \wedge d \psi_\ell(x,t)
= v_0 \wedge w_0 + \chi(x,t) h^{\prime \prime}(\ell (x,t) \cdot \xi) (c_2 v_0 - c_1 w_0) \wedge \xi + O \left( \frac{1}{\ell} \right),
\end{equation}
which is of the form \eqref{Aim on potentials} if 
\begin{equation} \label{Potential condition}
(c_2 v_0 - c_1 w_0) \wedge \xi = \omega.
\end{equation}
This raises the question whether \eqref{Potential condition} can be solved for $c_1,c_2 \in \R$. Notice that if $\omega_0 \neq 0$, the answer is independent of the factors $v_0,w_0$ of $\omega_0 = v_0 \wedge w_0$.

It turns out that given general $\omega_0,\omega$ with $\omega_0 \wedge \omega = 0$, such $c_1,c_2$ do not always exist. (The canonical bad case is $\omega_0=v \wedge \xi, \omega=w
\wedge \xi$, as then  $(c_1 v+c_2 \xi) \wedge \xi= c_1 v \wedge \xi= w \wedge \xi$ if and only if $v$ is parallel to $w$).
Essentially, when \eqref{Potential condition} holds, the segment defined by $V_0$ and $V$ is good (the case $\omega_0=0$ yielding some additional cases). 


\begin{rem}In \eqref{Cancellation formula}, we use crucially the cancellation properties of the wedge product $d\varphi_\ell \wedge d\psi_\ell$ to overcome the nonlinearity of $P_B$ and $P_E$.
In fact, $d\varphi_\ell \wedge d\psi_\ell$ arises, up to a term $O(1/\ell)$, as pullbacks of the bivector $v_0 \wedge w_0$. In
other words, 
 \[(d\varphi_\ell,d\psi_\ell) = \Phi_\ell^*(d\varphi,d\psi),\]
where $\varphi(x,t) = (x,t) \cdot v_0$, $\psi(x,t) = (x,t) \cdot w_0$ and $\Phi_\ell(x,t) = x + \ell^{-1} h'(\ell x \cdot \xi) \zeta$ with $\zeta \cdot v_0 = c_1$ and $\zeta \cdot w_0 = c_2$. Note that the class of simple two-forms is closed under taking pull-backs with $\Phi \in C^\infty(\R^4;\R^4)$, as a consequence of the formula $\Phi^* (v \wedge w) = \Phi^* v \wedge \Phi^* w$.
\end{rem}

\subsection{States in Clebsch variables}
As a matter of fact, when we iterate the construction and apply convex integration we will be modifying $d\varphi$ and $d\psi$ instead of $d\varphi \wedge d\psi$. We will therefore use a separate notation in which we keep track of the factors forming a bivector:
\begin{equation}\label{Wnotation}
W = (u, S, v, w) \in \R^4 \times \R_{\operatorname{sym}}^{3 \times 3} \times \R^4 \times \R^4, \qquad V = p(W) \defeq (u, S, v \wedge w) \in \mathscr{M}.\end{equation}
The case $\omega_0 = 0$ is special as we will be able to construct potentials only when we interpret $0=0 \wedge 0$ .

\subsection{Good and bad $\Lambda$-segments} \label{The answer to Question on existence of potentials}
To start, we consider  simple two-forms $\omega_0 = v_0 \wedge w_0 \neq 0$ and $\omega = v \wedge w \neq 0$ with $\omega_0 \wedge \omega = 0$. Since $\omega$ is simple, there exists $\xi \in (\R^3 \setminus \{0\}) \times \R$ such that $\omega \wedge \xi = 0$. We study separately the case where $\omega$ and $\omega_0$ are parallel and the one in which they are not.

\begin{prop} \label{Potentials for easy Lambda segments}
If $\omega = k \omega_0 \neq 0$ for some $k \in \R$, then \eqref{Potential condition} is satisfied for some $c_1,c_2 \in \R$.
\end{prop}

\begin{proof}
Since $\omega \wedge \xi = k v_0 \wedge w_0 \wedge \xi = 0$, we may write $d_1 v_0 + d_2 w_0 + d_3 \xi = 0$ for some $d_1,d_2,d_3 \in \R$, not all zero. Since $v_0 \wedge w_0 \neq 0$, we have $d_3 \neq 0$, which implies that $\{d_1,d_2\} \neq \{0\}$ (since $\xi \neq 0$).  If $d_2 \neq 0$, set $c_1 = 0$ and $c_2 = -kd_3/d_2$: then $[c_2 v_0 - c_1 w_0] \wedge \xi = k v_0 \wedge w_0 = \omega$. The case $d_1 \neq 0$ is similar.
\end{proof}

\begin{prop} \label{Equivalent conditions for good Lambda segments}
Suppose $\omega_0 \neq 0$ and $\omega \neq 0$ satisfy $\omega_0 \wedge \omega = 0$ but $\omega$ is not a multiple of $\omega_0$. The following conditions are equivalent.
\begin{enumerate}[\upshape (i)]
\item There exist $c_1,c_2 \in \R$ such that \eqref{Potential condition} holds.

\item $\omega \wedge \xi = 0$ but $\omega_0 \wedge \xi \neq 0$.

\item There exist $\tilde{v}, \tilde{w}_0 \in \R^4 \setminus \{0\}$ such that $\omega_0 = \tilde{v} \wedge \tilde{w}_0$ and $\omega = \tilde{v} \wedge \xi$.
\end{enumerate}
\end{prop}

\begin{proof}[Proof of (i) $\implies$ (ii)]
Suppose (i) holds and fix $c_1$ and $c_2$. Then $\omega \wedge \xi = 0$.  Seeking contradiction, assume $\omega_0 \wedge \xi = 0$. Then there exist constants $d_1,d_2,d_3 \in \R$, not all zero, such that $d_1 v_0 + d_2 w_0 + d_3 \xi = 0$. If $d_3 = 0$, then $v_0$ and $w_0$ are linearly dependent, which gives a contradiction with $\omega_0 = v_0 \wedge w_0 \neq 0$. On the other hand, if $d_3 \neq 0$, then $\xi \in \spann\{v_0,w_0\}$ and thus $\omega = (c_2 v_0 - c_1 w_0) \wedge \xi$ is a multiple of $\omega_0 = v_0 \wedge w_0$, giving a contradiction.
\end{proof}

\begin{proof}[Proof of (ii) $\implies$ (iii)]
By Proposition \ref{Characterisation of Lambda segments}, we can write $\omega = \tilde{v} \wedge \xi$ and either $\omega_0 = \tilde{v_0} \wedge \xi$ or $\omega_0 = \tilde{v} \wedge \tilde{w}_0$. The latter condition must then hold in view of (ii).
\end{proof}

\begin{proof}[Proof of (iii) $\implies$ (i)]
By assumption, $\omega_0 = v_0 \wedge w_0 = \tilde{v} \wedge \tilde{w}_0$. Thus $\tilde{v} \in \text{span}\{v_0,w_0\}$. Writing $\tilde{v} = c_2 v_0 - c_1 w_0$ we obtain $[c_2 v_0 - c_1 w_0] \wedge \xi = \tilde{v} \wedge \xi = \omega$.
\end{proof}



Thus we are ready to define a class of $\Lambda$-segments for which there exist the desired compactly supported plane waves (which are constructed in Proposition \ref{Proposition on existence of potentials}). We then define the corresponding $\Lambda_g$-convexity notions needed in the sequel.

\begin{defin}
Suppose $V_0 \defeq (u_0,S_0,\omega_0) \in \mathscr{M}$, $V \defeq (u,S,\omega) \in \Lambda$ and $0 < \lambda < 1$. We say that
\[[V_0 - (1-\lambda) V, V_0 + \lambda V] \quad \text{is a \emph{good $\Lambda$-segment ($\Lambda_g$-segment)}} \]
if there exists $\xi \in (\R^3 \setminus \{0\}) \times \R$ such that \eqref{Wave cone condition 1}--\eqref{Wave cone condition 3} and one of the conditions
\begin{align}
& \omega = 0, \label{Good Lambda condition 1} \\
& \omega_0 \wedge \xi \neq 0, \label{Good Lambda condition 2} \\
& \omega = k \omega_0 \neq 0, \quad k \in \R \setminus \{-1/\lambda,1/(1-\lambda)\}, \label{Good Lambda condition 3} \\
& u = S = \omega_0 = 0 \label{Good Lambda condition 4}
\end{align}
holds. Otherwise we say that $[V_0 - (1-\lambda) V, V_0 + \lambda V]$ is a \emph{bad $\Lambda$-segment}.
\end{defin}

The restriction on $k \in \R$ in \eqref{Good Lambda condition 3} ensures that the endpoints $V_0 - (1-\lambda) V$ and $V_0 + \lambda V$ have non-vanishing $\omega$-components; this is used in Propositions \ref{Proposition on existence of potentials} and \ref{Proposition on approximation of laminates}.

We define a lamination convex hull in terms of  $\Lambda_g$-segments.

\begin{defin}
Let $Y \subset \mathscr{M}$. We define the sets $Y^{k,\Lambda_g}$, $k \in \mathbb{N}_0$, as follows:
\begin{enumerate}[\upshape (i)]
\item $Y^{0,\Lambda_g} \defeq Y$.

\item If $k \ge 1$ and $V_0 \in \mathscr{M}$, the point $V_0$ belongs to $Y^{k,\Lambda_g}$ if $V_0 \in Y^{k-1,\Lambda_g}$ or there exist $\lambda \in (0,1)$ and $V \in \mathscr{M}$ such that $[V_0 - (1-\lambda) V, V_0 + \lambda V] \subset \mathscr{M}$ is a good $\Lambda$-segment whose endpoints belong to $Y^{k-1,\Lambda_g}$.
\end{enumerate}

\noindent Furthermore, we denote $Y^{lc,\Lambda_g} \defeq \cup_{k \in \N_0} Y^{k,\Lambda_g}$.
\end{defin}

We also give a related notion for finite-order laminates; recall Remark \ref{Remark on lamination convex hull}.

\begin{defin} \label{Definition of good laminates}
Suppose $\nu = \sum_{\mathbf{j} \in \{1,2\}^N} \mu_\mathbf{j} \delta_{V_\mathbf{j}}$ is a finite-order laminate supported in $Y \subset \mathscr{M}$. We say that $\nu$ is a \emph{good finite-order laminate}, and denote $\nu \in \mathcal{L}_g(Y)$, if for all $\mathbf{j}' \in \{1,2\}^k$, $1 \le k \le N-1$, the $\Lambda$-segment $[V_{\mathbf{j}',1},V_{\mathbf{j}',2}]$ is good.
\end{defin}

\subsection{Localised plane waves along $\Lambda_g$ segments} \label{Potentials for good Lambda-segments}
To every $\Lambda_g$-segment there corresponds a potential, and thus we can localise the plane waves. 

\begin{prop} \label{Proposition on existence of potentials}
Let $W_0 = (u_0,S_0,v_0, w_0) \in \R^{17}$, and suppose $[V_0-(1-\lambda)\bar{V},V_0+\lambda \bar{V}] \subset \mathscr{M}$ is a $\Lambda_g$-segment. If $\omega_0 = 0$, then suppose $v_0 = w_0 = 0$. Fix a cube $Q \subset \R^4$ and let $\varepsilon > 0$.

There exist $W_\ell \defeq W_0 + (\bar{u}_\ell,\bar{S}_\ell,d\bar{\varphi}_\ell, d\bar{\psi}_\ell) \in W_0 + C_c^\infty(Q;\R^{17})$ with the following properties.
\begin{enumerate}[\upshape (i)]

\item $\mathcal{L}(V_\ell) = 0$, where $V_\ell = p(W_\ell)$.

\item For every $(x,t) \in Q$ there exists $\tilde{W} = \tilde{W}(x,t) \in \R^{17}$ such that
\begin{align*}
& \tilde{V} = p(\tilde{W}) \in [V_0 - (1-\lambda) \bar{V}, V_0 + \lambda \bar{V}], \\
& |W_\ell(x,t) - \tilde{W}| < \varepsilon, \quad |V_\ell(x,t) - \tilde{V}| < \varepsilon.
\end{align*}

\item For every $\ell \in \N$ there exist pairwise disjoint open sets $A_1,A_2 \subset Q$ such that
\begin{align*}
& V_\ell(x,t) = V_0 + \lambda \bar{V} \text{ in } A_1 \text{ with } \abs{A_1} > (1-\varepsilon) (1-\lambda) \abs{Q}, \\
& V_\ell(x,t) = V_0 -(1-\lambda) \bar{V} \text{ in } A_2 \text{ with } \abs{A_2} > (1-\varepsilon) \lambda \abs{Q}.
\end{align*}
Furthermore, $W_\ell$ is locally constant in $A_1$ and $A_2$. For $j = 1,2$, writing $W_\ell = (u_j,S_j,v_j \wedge w_j)$ in $A_j$, we have either $v_j = w_j = 0$ or $v_j \wedge w_j \neq 0$.

\item $V_\ell \rightharpoonup V$ in $L^2(Q;\R^{15})$.

\end{enumerate}
\end{prop}

For the proof we first specify the oscillating functions that we intend to use. Their first derivatives can be chosen to be mollifications of 1-periodic sawtooth functions.

\begin{lem} \label{Lemma on sawtooth functions}
Suppose $0 < \lambda < 1$ and $\varepsilon > 0$. Then there exists $h \in C^\infty(\R)$ with the following properties:

\begin{enumerate}[\upshape (i)]
\item $h^{\prime \prime}$ is 1-periodic.

\item $-(1-\lambda) \le h^{\prime \prime} \le \lambda$.

\item $\int_0^1 h^{\prime \prime}(s) \, ds = 0$. (Thus, $h'$ is 1-periodic.)

\item $\abs{\{s \in [0,1] \colon h^{\prime \prime}(s) = \lambda\}} \ge (1-\varepsilon)(1-\lambda)$.

\item $\abs{\{s \in [0,1] \colon h^{\prime \prime}(s) = -(1-\lambda)\}} \ge (1-\varepsilon)\lambda$.
\end{enumerate}
\end{lem}

\begin{proof}[Proof of Proposition \ref{Proposition on existence of potentials}, the cases \eqref{Good Lambda condition 1}--\eqref{Good Lambda condition 3}]
Suppose one of the conditions \eqref{Good Lambda condition 1}--\eqref{Good Lambda condition 3} holds. Define the perturbation $(\bar{u}_\ell,\bar{S}_\ell,d\bar{\varphi}_\ell, d\bar{\psi}_\ell)$ via Lemma \ref{l:fluidpotential} and \eqref{Potential Ansatz 1}--\eqref{Potential Ansatz 2}. Claims (i) and (iv) are clear. In (ii) we choose $\tilde{W} = W_0 + \chi(x,t) h^{\prime \prime}((\ell (x,t) \cdot \xi) \bar{W}$.

In (iii) let $\varepsilon > 0$, fix a cube $\tilde{Q} \subset Q$ with $|\tilde{Q}| > (1-\varepsilon/3) \abs{Q}$ and choose $\chi$ such that $\chi = 1$ in $\tilde{Q}$. Cover $\tilde{Q}$, up to a set of measure $\varepsilon \abs{Q}/3$, by cubes $Q_1,\ldots,Q_N$ with one of the sides parallel to $\xi$. We wish to show that $
\abs{\{y \in Q_k \colon h^{\prime \prime}(\ell y \cdot \xi) = \lambda\}}
\ge \abs{Q_k} (1-\varepsilon/3) (1-\lambda)$ for every large enough $\ell \in \N$; in (iii) we may then choose $A_1 = \cup_{k=1}^N \{y \in Q_k \colon h^{\prime \prime}(\ell y \cdot \xi) = \lambda\}$. Similarly, $A_2 = \cup_{k=1}^N \{y \in Q_k \colon h^{\prime \prime}(\ell y \cdot \xi) = 1-\lambda\}$.

Choose an orthonormal basis $\{f_1,f_2,f_3,f_4\}$ of $\R^4$ such that $f_1 = \xi/\abs{\xi}$ and
\[
Q_k = \{y \in \R^4 \colon \zeta \cdot f_j \le y \cdot f_j \le \zeta \cdot f_j + l(Q_k)\}
\]
for some $\zeta \in \R^4$. In order to switch to coordinates where $Q_k$ has sides parallel to coordinate axes, define $L \defeq \sum_{j=1}^4 e_j \otimes f_j \in \R^{4 \times 4}$, so that $L f_j = e_j$ for $j = 1,\dots,4$ and therefore $L \in O(4)$. Then, denoting $z = L y$,
\begin{align*}
   Q_k
&= \{y \in \R^4 \colon L \zeta \cdot e_j \le L y \cdot e_j \le L \zeta \cdot e_j + l(Q_k)\} \\
&= L^{-1} \{z \in \R^4 \colon L \zeta \cdot e_j \le z \cdot e_j \le L \zeta \cdot e_j + l(Q_k)\} \\
&= L^{-1} \left( \prod_{j=1}^4 [(L\zeta)_j,(L\zeta)_j + l(P)] \right) = L^{-1}(L Q_k).
\end{align*}
Thus
\begin{align*}
   \abs{\{y \in Q_k \colon h^{\prime \prime}(\ell y \cdot \xi) = \lambda\}}
= &\abs{L^{-1} \{z \in L Q_k \colon h^{\prime \prime}(\ell \abs{\xi} z_1) = \lambda\}} \\
= &l(Q_k)^3 |\{s \in [(L\zeta)_1,(L\zeta)_1 + l(Q_k)] \colon \\
&h^{\prime \prime}(\ell \abs{\xi} z_1) = \lambda)\}| \\
\ge &\abs{Q_k} (1-\varepsilon/3) (1-\lambda)
\end{align*}
for all large $\ell \in \N$.

To finish the proof of (iii), write $W_\ell = (u_j,S_j,v_j \wedge w_j)$ in $A_j$, where $j \in \{1,2\}$. In the case \eqref{Good Lambda condition 1}, if $\omega_0 = 0$, then $v_j = w_j = 0$ by assumption, and if $\omega_0 \neq 0$, then $\omega_0 -(1-\lambda) \bar{\omega} \neq 0$ and $\omega_0 + \lambda \bar{\omega} \neq 0$. Next, in the case \eqref{Good Lambda condition 2}, by \eqref{waveconebivectors} we have $(\omega_0 + t \bar{\omega}) \wedge \xi = \omega_0 \wedge \xi \neq 0$, hence in particular $\omega_0 -(1-\lambda) \bar{\omega} \neq 0$ and $\omega_0 + \lambda \bar{\omega} \neq 0$. Finally, the case \eqref{Good Lambda condition 3} follows from the restriction $k \notin \{1/\lambda,1/(1-\lambda)\}$.
\end{proof}

The case \eqref{Good Lambda condition 4} requires a separate argument since in this case, \eqref{Potential condition} has no solutions $c_1,c_2 \in \R$. In fact, if $\lambda = 1/2$, we let $d\bar{\varphi}_\ell$ and $d\bar{\psi}_\ell$ oscillate in \emph{different} directions, and thus $W_\ell$ is not a plane wave. However, $V_\ell = p(W_\ell)$ oscillates along the $\Lambda_g$-segment $[-V/2,V/2]$. The general case $\lambda \in (0,1)$ then follows by combining with the case \eqref{Good Lambda condition 3}.

\begin{proof}[Proof of Proposition \ref{Proposition on existence of potentials}, the case \eqref{Good Lambda condition 4}]
The case $\bar{\omega} = 0$ being obvious, assume $\bar{\omega} = \bar{v} \wedge \bar{w} \neq 0$. Suppose first $\lambda = 1/2$.

Without loss of generality, assume $\bar{v} \cdot \bar{w} = 0$. Let $\varepsilon > 0$ and choose $\tilde{Q} \subset Q$ and $\chi$ as above. Then
\[\bar{\varphi}_\ell(x,t) \defeq \ell^{-1} \chi(x,t) h'((x,t) \cdot \ell \bar{v}), \qquad \bar{\psi}_\ell(x,t) \defeq 2 \ell^{-1} \chi(x,t) h'((x,t) \cdot \ell \bar{w})\]
have the sought properties for all large enough $\ell \in \N$.

Indeed, for (ii) choose $\tilde{W} = (0,0,\chi(x,t) h^{\prime \prime}((x,t) \cdot \ell \bar{v}) \bar{v}, 2 \chi(x,t) h^{\prime \prime} ((x,t) \cdot \ell \bar{w}) \bar{w})$. For (iii), note that when $(x,t) \in \tilde{Q}$, we have
\begin{equation} \label{V_l in good set}
V_\ell(x,t) = \begin{cases}
                  V_0 + 2^{-1} (0,0,\bar{v} \wedge \bar{w}) & \text{when } h^{\prime \prime}((x,t) \cdot \ell \bar{v}) = h^{\prime \prime}((x,t) \cdot \ell \bar{w}) = \pm 2^{-1}, \\
                  V_0 - 2^{-1} (0,0,\bar{v} \wedge \bar{w}) & \text{when } h^{\prime \prime}((x,t) \cdot \ell \bar{v}) = - h^{\prime \prime}((x,t) \cdot \ell \bar{w}) = \pm 2^{-1}.
                \end{cases}
\end{equation}
Cover $\tilde{Q}$ up to a small set by cubes $Q_1,\ldots,Q_N$ with two sides parallel to $\bar{v}_x$ and $\bar{w}_x$; recall that $\bar{v}_x \cdot \bar{w}_x = 0$.

For $k = 1,\ldots,N$ we get $\{(x,t) \in Q_k \colon  h^{\prime \prime}((x,t) \cdot \ell \bar{v}) = h^{\prime \prime}((x,t) \cdot \ell \bar{w}) = 2^{-1}\} \ge \abs{\{s \in [0,1] \colon h^{\prime \prime}(s) = 2^{-1}\}}^2 \abs{Q_k} - O(1/\ell) > \abs{Q_k} (1-\varepsilon/3) (1/2)^2$ as in the previous proof, and a similar inequality holds for the other three cases of \eqref{V_l in good set}. This completes the proof of the case $\lambda = 1/2$.

We then cover the case $\lambda \neq 1/2$. Let $0 < \delta < \min \{\lambda,1-\lambda\}$. Using the case above, we choose $d\bar{\varphi}_\ell$, $d\bar{\psi}_\ell$ satisfying claims (i)--(iv) for $[V_0 - \delta \bar{V}, V_0 + \delta \bar{V}]$ and $\varepsilon/2$. Note that $d\bar{\varphi}_\ell \wedge d\bar{\psi}_\ell = \delta \bar{\omega} \neq 0$ in $A_1$ and $d\bar{\varphi}_\ell \wedge d\bar{\psi}_\ell = -\delta \bar{\omega} \neq 0$ in $A_2$. We then cover the sets $A_1$ and $A_2$ by cubes up to a small set and apply the case \eqref{Good Lambda condition 3} in the cubes. (The last claim of (iii) is clear.)
\end{proof}

\begin{rem}
We have looked for solutions of \eqref{Aim on potentials} of the form \eqref{Potential Ansatz 1}--\eqref{Potential Ansatz 2}, and in some special cases, a solution does not exist. It is conceivable that another Ansatz would satisfy \eqref{Aim on potentials} in some of cases the cases excluded by \eqref{Potential Ansatz 1}--\eqref{Potential Ansatz 2}. This would essentially require a degenerate Darboux Theorem with a Dirichlet boundary condition -- more concretely, solving $d\varphi_\ell \wedge d\psi_\ell = v_0 \wedge w_0 + h^{\prime \prime}(\ell (x,t) \cdot \xi) v \wedge w + O(1/\ell)$ with $(d\varphi_\ell,d\psi_\ell) = (v_0,w_0)$ on $\partial Q$. However, such theorems are remarkably difficult to prove and to the authors' knowledge, a suitable existence result is not available at this point; we refer to~\cite[\textsection 14]{CDK} and the references contained therein.
\end{rem}

\section{Characterisations of the relative interior of the lamination convex hull} \label{Characterisations of the relative interior of the lamination convex hull}
Our next task is to find a suitable (relatively open) set $\U_{r,s} \subset \text{int}_{\mathscr{M}}(K_{r,s}^{lc,\Lambda})$ where to run convex integration. Since we have only constructed potentials for $\Lambda_g$-segments, we would like to produce $\U_{r,s}$ by using $\Lambda_g$-segments only. The choice of $\U_{r,s}$ is, however, non-trivial, as discussed in \textsection \ref{A rigidity result on the good Lambda-hull}. Nevertheless, eventually the following simple definition turns out to suffice.
\begin{defin}
We denote
\[\U_{r,s} \defeq \text{int}_{\mathscr{M}}(K_{r,s}^{lc,\Lambda}).\]
\end{defin}
In the main result of this chapter, Theorem \ref{Hull theorem}, we give several characterisations of $\U_{r,s}$ and show, in particular, that $0 \in \U_{r,s}$.

\subsection{A rigidity result on the good $\Lambda$-hull} \label{A rigidity result on the good Lambda-hull}
Initially, it appears natural to choose some set $\U_{r,s} \subset K_{r,s}^{lc,\Lambda_g}$ for strict subsolutions. However, $K_{r,s}^{lc,\Lambda_g}$ turns out to be rather small; in fact,
\[
E_0 = B_0 \times u_0 \qquad \text{for every } V_0 = (u_0,S_0,B_0,E_0) \in K^{lc,\Lambda_g}.
\]

\begin{prop} \label{Proposition on rigidity of good Lamdba}
Suppose $[V_0 - (1-\lambda) V, V_0 + \lambda V] \subset \mathscr{M}$ is a $\Lambda_g$-segment, and assume that $V_1 \defeq V_0 + \lambda V$ and $V_2 \defeq V_0 - (1-\lambda) V$ satisfy $E_j = B_j \times u_j$. Then $E_0 = B_0 \times u_0$.
\end{prop}

The proof consists of two parts. First, the $\Lambda_g$-conditions and the assumption $E_j = B_j \times u_j$ lead to the conclusion $(B_1-B_2) \times (u_1-u_2) = 0$. Then a bit of algebra gives
\[\lambda B_1 \times u_1 + (1-\lambda) B_2 \times u_2 = (\lambda B_1 + (1-\lambda) B_2) \times (\lambda u_1 + (1-\lambda) u_2),\]
that is, $E_0 = B_0 \times u_0$.

At first sight, Proposition \ref{Proposition on rigidity of good Lamdba} seems to prevent convex integration unless potentials are found for bad $\Lambda$-segments. However, this rigidity disappears once one considers $\Lambda_g$-convex hulls of \emph{relatively open} sets. Indeed, whenever $\U$ is bounded and relatively open in $\mathscr{M}$, we have $\U^{lc,\Lambda_g} = \U^{lc,\Lambda}$ (see Proposition \ref{Theorem on hulls of relatively open sets}). The basic reason behind this phenomenon is the fact that, loosely speaking, bad $\Lambda$-segments become good when translated to almost any direction.

\subsection{Laminates of relatively open sets in $\mathscr{M}$}
We start the proof of Proposition \ref{Theorem on hulls of relatively open sets} by showing that the class of relatively open sets in $\mathscr{M}$ is closed with respect to taking laminates:

\begin{prop} \label{Lemma on laminates of relatively open sets}
Suppose $\U$ is relatively open in $\mathscr{M}$. Then $U^{lc,\Lambda_g}$ is relatively open in $\mathscr{M}$.
\end{prop}

Before beginning the proof of Proposition \ref{Lemma on laminates of relatively open sets} we describe the main difficulty. The proof proceeds by induction. Suppose $V_0 - (1-\lambda) V, V_0 + (1-\lambda) V \in \U^{k,\Lambda_g}$, $[V_0 - (1-\lambda) V, V_0 + \lambda V] \subset \mathscr{M}$ is a $\Lambda_g$-segment and $B_{\mathscr{M}}(V_0+ \lambda V,\delta) \cup B_{\mathscr{M}}(V_0 - (1-\lambda) V, \delta) \subset \U^{lc,\Lambda_g}$. Given $\tilde{V}_0 \in \mathscr{M}$ with $|V_0 - \tilde{V}_0|$ small our aim is to get $\tilde{V}_0 \in \U^{lc,\Lambda_g}$. It is tempting to write $\tilde{V}_0 = \lambda [\tilde{V}_0 - (1-\lambda) V] + (1-\lambda) [\tilde{V}_0 + \lambda V]$.

It is, however, not guaranteed that the endpoints $\tilde{V}_0 + \lambda V, \tilde{V}_0 - (1-\lambda) V$ lie on the nonlinear manifold $\mathscr{M}$! Therefore, we need to perturb $\tilde{V}_0 + \lambda V$ and $\tilde{V}_0 - (1-\lambda) V$ in order to place an entire $\Lambda_g$-segment on $\mathscr{M}$. This is in stark contrast to equations of fluid dynamics where the lamination convex hull has non-empty interior. Again, the two-form formalism comes to the rescue. 

We overcome the difficulties via the following lemma which allows us to choose the factors $v,w \in \R^4$ of a simple two-form $v \wedge w$ in a continuous way. Henceforth, we denote $\norm{\omega} \defeq \max_{\abs{f} = \abs{g} = 1} \omega(f,g)$ for every $\omega \in \Lambda^2(\R^4)$.

\begin{lem} \label{Lemma about tubes}
Suppose $v_1,w_1,v_2,w_2 \in S^3$ with $v_1 \cdot w_1 = v_2 \cdot w_2 = 0$, and let $0 < \varepsilon < 1$. If $\norm{v_1 \wedge w_1 - v_2 \wedge w_2} < \varepsilon$, then there exist orthogonal $\tilde{v}_2,\tilde{w}_2 \in S^3$ such that
\[\tilde{v}_2 \wedge \tilde{w_2} = v_2 \wedge w_2, \; \abs{v_1-\tilde{v}_2} < \sqrt{2} \varepsilon \; \text{ and } \; |w_1-\tilde{w}_2| < \sqrt{2} \varepsilon.\]
\end{lem}

\begin{proof}
First, if $v_2 \cdot v_1 = w_2 \cdot v_1 = 0$, then $(v_1 \wedge w_1 - v_2 \wedge w_2)(v_2,w_2) = -1$, which yields a contradiction.
Assume, therefore, that $v_2 \cdot v_1$ and $w_2 \cdot v_1$ are not both zero.

Denote by $\tilde{v}_2$ the normalised projection of $v_1$ onto $\text{span}\{v_2,w_2\}$ and by $\tilde{w}_2$ its rotation in $\text{span}\{v_2,w_2\}$, that is,
\[\tilde{v}_2 = \frac{(v_1 \cdot v_2) v_2 + (v_1 \cdot w_2) w_2}{\abs{(v_1 \cdot v_2) v_2 + v_1 \cdot w_2) w_2}}, \qquad\tilde{w}_2 = \frac{-(v_1 \cdot w_2) v_2 + (v_1 \cdot v_2) w_2}{\abs{(v_1 \cdot v_2 ) v_2 + (v_1 \cdot w_2) w_2}}.\]
Thus $\tilde{v}_2 \wedge \tilde{w}_2 = v_2 \wedge w_2$ and $\tilde{w}_2 \cdot v_1 = 0$. Now
\[(v_1 \wedge w_1 - \tilde{v}_2 \wedge \tilde{w}_2) \left( \frac{\tilde{v}_2 - (\tilde{v}_2 \cdot v_1) v_1}{\abs{\tilde{v}_2 - (\tilde{v}_2 \cdot v_1) v_1}},\tilde{w}_2 \right) = - \frac{1 - (\tilde{v}_2 \cdot v_1)^2}{\sqrt{1 - (\tilde{v}_2 \cdot v_1)^2}} = - \sqrt{1 - (\tilde{v}_2 \cdot v_1)^2}.\]
Thus $\sqrt{1 - (\tilde{v}_2 \cdot v_1)^2} \le \norm{v_1 \wedge w_1 - \tilde{v}_2 \wedge \tilde{w}_2} < \varepsilon$. Since clearly $\tilde{v}_2 \cdot v_1 \ge 0$, we conclude that $\tilde{v}_2 \cdot v_1 > \sqrt{1 - \varepsilon^2}$. Hence, $\abs{v_1-\tilde{v}_2}^2 < 2 - 2 \sqrt{1-\varepsilon^2} < 2 \varepsilon^2$.

We then show that $|w_1-\tilde{w}_2| < \sqrt{2} \varepsilon$. First,
\[(v_1 \wedge w_1 - \tilde{v}_2 \wedge \tilde{w}_2) \left( v_1, \frac{w_1 - (\tilde{w}_2 \cdot w_1)\tilde{w}_2}{|w_1 - (\tilde{w}_2 \cdot w_1) \tilde{w}_2|} \right)
= \sqrt{1 - (\tilde{w}_2 \cdot w_1)^2}\]
gives $\sqrt{1 - (\tilde{w}_2 \cdot w_1)^2} < \varepsilon$. Next,
\[(v_1 \wedge w_1 - \tilde{v}_2 \wedge \tilde{w}_2) (v_1, w_1) = 1 - (v_1 \cdot \tilde{v}_2) (\tilde{w}_2 \cdot w_1) < \varepsilon\]
implies that $\tilde{w}_2 \cdot w_1 > 0$. As above, we conclude that $|w_1 - \tilde{w}_2|^2 < 2 \varepsilon^2$.
\end{proof}

We also need a lemma which gives a solution of a matrix equation with a natural norm estimate.

\begin{lem} \label{Lemma on good solutions of matrix equations}
If $x \in \R^3 \setminus \{0\}$ and $y \in \R^3$, then
\[S \defeq \frac{x \otimes y + y \otimes x - (x \cdot y) I}{\abs{x}^2} \in \R^{3 \times 3}_{\operatorname{sym}}\]
satisfies $S x = y$ and $\abs{S} \le 3 \abs{y}/\abs{x}$.
\end{lem}

\begin{proof}[Proof of Proposition \ref{Lemma on laminates of relatively open sets}]
We need to show for every $k \in \N_0$ that if $V_0 \in \U^{k,\Lambda_g}$, then there exists $\delta > 0$ such that $B_{\mathscr{M}}(V_0,\delta) \subset \U^{lc,\Lambda_g}$. The claim holds for $k = 0$ by assumption, so assume, by induction, it holds for $k$.

Let $V_0 \in \U^{k+1,\Lambda_g}$. Write $V_0 = \lambda (V_0 - (1-\lambda) V) + (1-\lambda) (V_0 + \lambda V)$, where $[V_0 - (1-\lambda) V, V_0 + \lambda V]$ is a $\Lambda_g$-segment. By assumption, there exists $\delta > 0$ such that
\[B_{\mathscr{M}}(V_0-(1-\lambda)V,\delta) \cup B_{\mathscr{M}}(V_0 + \lambda V,\delta) \subset \U^{lc,\Lambda_g}.\]
We intend to show that whenever $\tilde{\delta} = \tilde{\delta}_{V_0,V,\lambda} > 0$ is small enough, $B_{\mathscr{M}}(V_0,\tilde{\delta}) \subset \U^{lc,\Lambda_g}$. The case \eqref{Good Lambda condition 1} is clear.

\vspace{0.3cm}
Suppose first \eqref{Good Lambda condition 2} holds, that is, $\omega \wedge \xi = 0$ but $\omega_0 \wedge \xi \neq 0$. By Proposition \ref{Equivalent conditions for good Lambda segments} and scaling, we may write $V_0 = (u_0,S_0, \norm{\omega_0} v \wedge w_0)$ and $V = (u,S, v \wedge \xi)$, where $\abs{v} = \abs{w_0} = 1$ and $v \cdot w_0 = 0$.

Let now $\tilde{V}_0 = (\tilde{u}_0, \tilde{S}_0, \tilde{\omega}_0) \in \mathscr{M}$ and $\|\tilde{V}_0-V_0\| < \tilde{\delta}$. By Lemma \ref{Lemma about tubes}, we may write $\tilde{\omega}_0/\|\tilde{\omega}_0\| = \tilde{v} \wedge \tilde{w}_0$, where $|\tilde{v}| = |\tilde{w}_0| = 1$, $\tilde{v} \cdot \tilde{w}_0 = 0$ and $|\tilde{v}-v| + |\tilde{w}_0 - w_0| \lesssim_{V_0} \tilde{\delta}$. In the last estimate we used the inequality $\|\omega_0/\norm{\omega_0} - \tilde{\omega}_0/\|\tilde{\omega}_0\|\| \le 2 \|\omega_0-\tilde{\omega}_0\|/\norm{\omega_0}$.

Now choose $\tilde{V} = (u,S, \tilde{v} \wedge \xi) \in \Lambda$. As long as $\tilde{\delta} > 0$ is small enough, it is ensured that $\tilde{v} \wedge \tilde{w}_0 \wedge \xi \neq 0$, so that $[\tilde{V}_0 - (1-\lambda) \tilde{V}, \tilde{V}_0 + \lambda \tilde{V}]$ satisfies \eqref{Good Lambda condition 2}. Thus $\tilde{V}_0 = \lambda (\tilde{V}_0 - (1-\lambda) \tilde{V}) + (1-\lambda) (\tilde{V}_0 + \lambda \tilde{V}) \in \U^{lc,\Lambda_g}$, as claimed.

\vspace{0.3cm}
Suppose next \eqref{Good Lambda condition 3} holds, so that $\omega = k \omega_0 \neq 0$ with $k \notin \{-1-\lambda,1/(1-\lambda)\}$. Write $\omega_0 = \norm{\omega_0} v_0 \wedge \xi \neq 0$, where $\abs{v_0} = \abs{\xi} = 1$ and $v_0 \cdot \xi = 0$. Again, let $\tilde{V}_0 = (\tilde{u}_0, \tilde{S}_0, \tilde{\omega}_0) \in \mathscr{M}$ and $\|\tilde{V}_0-V_0\| < \tilde{\delta}$. This time, we may write $\tilde{\omega}_0 = \|\tilde{\omega}_0\| \tilde{v}_0 \wedge \tilde{\xi}$, where $|\tilde{v}_0| = |\tilde{\xi}| = 1$, $\tilde{v}_0 \cdot \tilde{\xi} = 0$ and $|\tilde{v}_0-v_0| + |\tilde{\xi}-\xi| \lesssim_{V_0} \tilde{\delta}$.

Our aim is to choose $\tilde{u} \approx u$ and $\tilde{S} \approx S$ such that $\tilde{V} = (\tilde{u}, \tilde{S}, k \tilde{\omega}_0)$ satisfies $\tilde{V} \tilde{\xi} = 0$. We select
\[\tilde{u} \defeq u - \frac{u \cdot \tilde{\xi}_x}{|\tilde{\xi}_x|^2} \tilde{\xi}_x\]
so that $\tilde{u} \cdot \tilde{\xi}_x = 0$ and $|\tilde{u}-u| \lesssim_{V_0,V,\lambda} \tilde{\delta}$ as soon as, say, $\delta < \abs{\xi_x}/2$. We then use Lemma \ref{Lemma on good solutions of matrix equations} to choose $\tilde{S} \in \mathbb{R}^{3 \times 3}_{\text{sym} }$ satisfying
\[\tilde{S} \tilde{\xi}_x + \tilde{\xi}_t \tilde{u}
= (\tilde{S}-S) \tilde{\xi}_x + S (\tilde{\xi}_x-\xi_x) + \tilde{\xi}_t (\tilde{u}-u) + (\tilde{\xi}_t - \xi_t) u = 0\]
with $|\tilde{S}-S| \lesssim_{V_0,V,\lambda} \tilde{\delta}$. Now $[\tilde{V}_0 - (1-\lambda) \tilde{V}, \tilde{V}_0 + \lambda \tilde{V}]$ satisfies \eqref{Good Lambda condition 3} and the endpoints belong to $\U^{lc,\Lambda_g}$. We conclude that $\tilde{V}_0 \in \U^{lc,\Lambda_g}$.

\vspace{0.3cm}
Last suppose $u = S = \omega_0 = 0 \neq \omega$. Let $\tilde{V}_0 = (\tilde{u}_0, \tilde{S}_0, \tilde{v}_0 \wedge \tilde{w}_0) \in \mathscr{M}$ with $|\tilde{V}_0-V_0| < \tilde{\delta}$. Suppose first $\tilde{v}_0 \wedge \tilde{w}_0 = 0$. Then $\tilde{V}_0 \in [\tilde{V}_0 - (1-\lambda) V, \tilde{V}_0 + \lambda V]$, the $\Lambda$-segment satisfies \eqref{Good Lambda condition 4} and the endpoints belong to $\U^{lc,\Lambda_g}$, so that $\tilde{V}_0 \in \U^{lc,\Lambda_g}$.

Suppose then $\tilde{v}_0 \wedge \tilde{w}_0 \neq 0$. We write $V = (0,0,\xi \wedge w)$ and choose
\[\tilde{V} = (0,0,\tilde{\xi} \wedge (\tilde{w} + \tilde{w}_0)) \in \Lambda,\]
where $\tilde{\xi}_x \neq 0$, $|\tilde{\xi} - \xi| + |\tilde{w} - w| < \tilde{\delta}$ and furthermore $\tilde{v}_0 \wedge \tilde{w}_0 \wedge \tilde{\xi} \neq 0$ and $\tilde{v}_0 \wedge \tilde{w}_0 \wedge \tilde{w} \neq 0$. Thus
 \[\tilde{V}_0 + (0,0,\tilde{v}_0 \wedge \tilde{w}) + \lambda \tilde{V} = (\tilde{u}_0, \tilde{S}_0, (\tilde{v}_0 + \lambda \xi) \wedge (\tilde{w}_0 + \tilde{w})) \in \U^{lc,\Lambda_g}\]
and $\tilde{V}_0 + (0,0,\tilde{v}_0 \wedge \tilde{w}) - (1-\lambda) \tilde{V} \in \U^{lc,\Lambda_g}$. Now $\tilde{V}_0 + (0,0, \tilde{v}_0 \wedge \tilde{w}) \in \U^{lc,\Lambda_g}$; the $\Lambda$-segment is good because $\tilde{V} \tilde{\xi} = 0$ but $\tilde{v}_0 \wedge \tilde{w}_0 \wedge \tilde{\xi} \neq 0$.
 
An entirely similar argument gives $\tilde{V}_0 - (0,0,\tilde{v}_0 \wedge \tilde{w}) \in \U^{lc,\Lambda_g}$. Now $[\tilde{V}_0 - (0,0,\tilde{v}_0 \wedge \tilde{w}), \tilde{V}_0 + (0,0,\tilde{v}_0 \wedge \tilde{w})]$ is a $\Lambda_g$-segment because we assumed that $\tilde{v}_0 \wedge \tilde{w}_0 \wedge \tilde{w} \neq 0$. Thus $\tilde{V}_0 \in \U^{lc,\Lambda_g}$, as claimed.
\end{proof}

\subsection{Equivalence of hulls of relatively open sets} \label{Equivalence of hulls of relatively open sets}

\begin{prop} \label{Theorem on hulls of relatively open sets}
Suppose $\U$ is bounded and relatively open in $\mathscr{M}$. Then $\U^{lc,\Lambda} = \U^{lc,\Lambda_g}$.
\end{prop}

\begin{proof}
The direction $\U^{lc,\Lambda_g} \subset \U^{lc,\Lambda}$ is obvious. We prove the converse direction by induction, first assuming that $(u,S,0) \notin \U^{lc,\Lambda}$ for each $u$ and $S$. Clearly $\U \subset \U^{lc,\Lambda_g}$. Assume, therefore, that $\U^{k,\Lambda} \subset \U^{lc,\Lambda_g}$; our aim is to show that $\U^{k+1,\Lambda} \subset \U^{lc,\Lambda_g}$.

Suppose $[V_0-(1-\lambda)V,V_0 + \lambda V] \subset \mathscr{M}$ is a bad $\Lambda$-segment and the endpoints $V_0-(1-\lambda)V,V_0 + \lambda V \in \U^{k,\Lambda} \subset \U^{lc,\Lambda_g}$. Assume first that $\omega_0 \neq 0$ and that $\omega_0$ and $\omega$ are not parallel. Thus $\omega_0 - (1-\lambda) \omega, \omega_0 + \lambda \omega \neq 0$. Now $\omega_0 = \xi \wedge w_0$ and $\omega = \xi \wedge w$ by Propositions \ref{Equivalent conditions for wave cone condition} and \ref{Equivalent conditions for good Lambda segments}. Choose $\tilde\omega \defeq \varepsilon w_0 \wedge w$ and $W = (0,0,\tilde\omega)$, where $\varepsilon \neq 0$ is small. Now, since $\omega_0 \wedge \tilde\omega = \omega \wedge \tilde\omega = 0$, we have $V_0 + W + \lambda V, V_0+W-(1-\lambda) V \in \mathscr{M}$. Proposition \ref{Lemma on laminates of relatively open sets} then gives $V_0 + W + \lambda V, V_0+W-(1-\lambda) V \in \U^{lc,\Lambda_g}$. Furthermore, $[V_0+W-(1-\lambda)V,V_0+W+\lambda V]$ is a $\Lambda_g$-segment because $\omega \wedge \xi = 0$ but $(\omega_0+\varepsilon w_0 \wedge w) \wedge \xi \neq 0$. Therefore $V_0 + W \in \U^{lc,\Lambda_g}$. Similarly, $V_0-W \in \U^{lc,\Lambda_g}$. Finally, $[V_0-W,V_0+W]$ is a $\Lambda_g$-segment because $\tilde\omega \wedge w = 0$ and yet $\omega_0 \wedge w \neq 0$. Consequently, $V_0 \in \U^{lc,\Lambda_g}$.

Assume next that $\omega_0 = 0$. Since $[V_0-(1-\lambda)V,V_0 + \lambda V] \subset \mathscr{M}$ is a bad $\Lambda$-segment, we have $\omega = \xi \wedge w \neq 0$. We may assume that $w_x \neq 0$ (by possibly adding a constant multiple of $\xi_x \neq 0$ to $w_x$). This time select a basis $\{\xi,w,f,g\}$ of $\R^4$ with $f_x \neq 0$ and $w_x \cdot f_x = 0$. Select $W = (0,0,\varepsilon \, w \wedge f)$ with $\varepsilon \neq 0$ small. Arguing as in the previous paragraph, $V_0 \pm W + \lambda V, V_0 \pm W - (1-\lambda) V \in \U^{lc,\Lambda_g}$. As above, $V_0 \pm W \in \U^{lc,\Lambda_g}$ since $(\omega_0 \pm w \wedge f) \wedge \xi \neq 0$. Now $[V_0-W,V_0+W]$ (with $\lambda = 1/2$) satisfies \eqref{Good Lambda condition 4}; thus $V_0 \in \U^{lc,\Lambda_g}$.

Finally assume $\omega_0 \neq 0$ and $\omega = k \omega_0$ for $k \in \{-1/\lambda,1/(1-\lambda)\}$. We may thus write $\omega_0 = v_0 \wedge \xi$. Choose $W = (0,0,v_0 \wedge w)$, where $v_0 \wedge w \wedge \xi \neq 0$; thus, after scaling $w$, $V_0 \pm W \in \U^{lc,\Lambda_g}$. Indeed, $V_0 + \lambda V \pm W \in \U^{lc,\Lambda_g}$ and $V_0 - \lambda V \pm W \in \U^{lc,\Lambda_g}$ by Proposition \ref{Lemma on laminates of relatively open sets}. The $\Lambda$-segment $[V_0 + \lambda V \pm W, V_0 - (1-\lambda) V \pm W]$ is good since $\omega \wedge\xi = 0$ but $(\omega_0 \pm v_0 \wedge w) \wedge \xi \neq 0$. Now the $\Lambda$-segment $[V_0-W,V_0+W]$ is good since $v_0 \wedge w \wedge w = 0$ but $v_0 \wedge \xi \wedge w \neq 0$. Thus, again, $V_0 \in \U^{lc,\Lambda_g}$.
\end{proof}

\subsection{Formulation of the characterisations}
Proposition \ref{Theorem on hulls of relatively open sets} allows us to use the whole wave cone $\Lambda$ in computations on hulls of relatively open sets. In order to exploit this, in Theorem \ref{Hull theorem} we characterise $\U_{r,s} \defeq \text{int}_{\mathscr{M}}(K_{r,s}^{lc,\Lambda})$ via different relatively open sets. Our main aim is twofold: first, $\U_{r,s} = \cup_{\tau \in [0,1)} (B_{\mathscr{M}}(K_{\tau r, \tau s},\varepsilon_\tau))^{lc,\Lambda_g}$ whenever the constants $\varepsilon_\tau > 0$ are small enough, and secondly, $0 \in \U_{r,s}$. We prove the first one via the (a priori) easier equality $\U_{r,s} = \cup_{\tau \in [0,1)} (B_{\mathscr{M}}(K_{\tau r, \tau s},\varepsilon_\tau))^{lc,\Lambda}$ and Proposition \ref{Theorem on hulls of relatively open sets}.

In order to prove both of our two aims in a unified manner, we introduce some further terminology. For every $u,B \in \R^3$ we denote
\[S_{u,B} \defeq u \otimes u - B \otimes B \in \R^{3 \times 3}_{sym}\]
and for every $c > 0$ we define relatively open sets
\begin{align*}
\V_{r,s,c}
\defeq &\{(u,S,B,E) \colon \abs{u+B} < r + c, \abs{u-B} < s + c, \abs{S - S_{u,B} - \Pi I} < c, \\ & \abs{\Pi} < rs + c, \abs{E - B \times u} < c, B \cdot E = 0\}.
\end{align*}
Note that given $c > 0$ we have $0 \in \V_{r,s,c}$ and $B_{\mathscr{M}}(K_{r,s},\tilde{c}) \subset \V_{r,s,c}$ for every small enough $\tilde{c} > 0$.

\begin{thm} \label{Hull theorem}
There exist constants $\varepsilon_\tau = \varepsilon_{\tau,r,s} > 0$ such that for any $\tau_0 \in (0,1)$,
\[\U_{r,s} = \bigcup_{\tau_0 < \tau < 1} \V_{\tau r,\tau s,\varepsilon_\tau}^{lc,\Lambda_g} = \bigcup_{\tau_0 < \tau < 1} (B_{\mathscr{M}}(K_{\tau r, \tau s},\varepsilon_\tau))^{lc,\Lambda_g} = \bigcup_{\tau_0 < \tau < 1} K_{\tau r, \tau s}^{lc,\Lambda}.\]
\end{thm}

We divide the proof of Theorem \ref{Hull theorem} into two propositions. 

\begin{prop} \label{Proposition for Hull theorem 1}
For every $\tau \in [0,1)$ there exists $\varepsilon_\tau > 0$ such that $\U_{r,s} \supset \V_{\tau r, \tau s, \varepsilon_\tau}$.
\end{prop}

\begin{prop} \label{Proposition for Hull theorem 3}
$\U_{r,s} \subset \cup_{\tau_0 < \tau < 1} K^{lc,\Lambda}_{\tau r,\tau s}$ for every $\tau_0 \in (0,1)$.
\end{prop}

Propositions \ref{Proposition for Hull theorem 1}--\ref{Proposition for Hull theorem 3} are proved in the rest of this chapter. Assuming Propositions \ref{Proposition for Hull theorem 1}--\ref{Proposition for Hull theorem 3}, Theorem \ref{Hull theorem} is obtained as follows:

\begin{proof}[Proof of Theorem \ref{Hull theorem}]
Whenever $0 < \tau_0 < 1$ and the constants $\varepsilon_\tau > 0$ are small enough, Propositions \ref{Theorem on hulls of relatively open sets} and \ref{Proposition for Hull theorem 1} give $K^{lc,\Lambda}_{r,s} \supset \cup_{\tau_0 < \tau < 1} \V_{\tau r, \tau s, \varepsilon_\tau}^{lc,\Lambda} = \cup_{\tau_0 < \tau < 1} \V_{\tau r, \tau s, \varepsilon_\tau}^{lc,\Lambda_g}$. Together with Proposition \ref{Lemma on laminates of relatively open sets}, which says that $\cup_{\tau_0 < \tau < 1} \V_{\tau r, \tau s, \varepsilon_\tau}^{lc,\Lambda_g}$ is relatively open in $\mathscr{M}$, this yields that $\U_{r,s} \supset \cup_{\tau_0 < \tau < 1} \V_{\tau r, \tau s, \varepsilon_\tau}^{lc,\Lambda_g}$ by the definition of $\U_{r,s}$.

Next, the inclusion $\cup_{\tau_0 < \tau < 1} \V_{\tau r,\tau s,\varepsilon_\tau}^{lc,\Lambda_g} \supset \cup_{\tau_0 < \tau < 1} K^{lc,\Lambda}_{\tau r,\tau s}$ follows directly from the fact that $\V_{\tau r,\tau s,\varepsilon_\tau} \supset K_{\tau r, \tau s}$ and Proposition \ref{Theorem on hulls of relatively open sets}. Proposition \ref{Proposition for Hull theorem 3} then says that $\U_{r,s} \subset \cup_{\tau_0 < \tau < 1} K^{lc,\Lambda}_{\tau r,\tau s}$.

Given parameters $\varepsilon_\tau > 0$ we choose $\tilde{\varepsilon}_\tau > 0$ such that $\V_{\tau r, \tau s, \varepsilon_\tau} \supset B_{\mathscr{M}}(K_{\tau r, \tau s},\tilde{\varepsilon}_\tau)$, and then $\U_{r,s} \supset \cup_{\tau_0 < \tau < 1}  B_{\mathscr{M}}(K_{\tau r, \tau s},\tilde{\varepsilon}_\tau)^{lc,\Lambda_g} \supset \cup_{\tau_0 < \tau < 1} K^{lc,\Lambda}_{\tau r,\tau s} \supset \U_{r,s}$. Theorem \ref{Hull theorem} holds for these adjusted parameters $\tilde{\varepsilon}_\tau > 0$.
\end{proof}

\subsection{Els\"{a}sser variables in relaxed MHD}
In some of the computations on relaxed MHD it will be convenient to replace the variables $(u,S,B,E)$ by Els\"{a}sser variables and a matrix component, $(z^+,z^-,M)$, which satisfy
\begin{align*}
& z^\pm = u \pm B, \qquad u = \frac{z^++z^-}{2}, \qquad B = \frac{z^+-z^-}{2}, \\
& M = S + A, \qquad M^T = S-A, \qquad S = \frac{M+M^T}{2}, \qquad A = \frac{M-M^T}{2}.
\end{align*}
The main advantage is that the constraint set obtains the particularly simple form
\[K_{r,s} = \{(z^+,z^-,z^+ \otimes z^- + \Pi I) \colon |z^+| = r, \; |z^-| = s, \; \abs{\Pi} \le rs\}.\]
The wave cone conditions \eqref{Wave cone condition 1}--\eqref{Wave cone condition 3} are written in Els\"{a}sser formalism as
\begin{equation} \label{Wave cone conditions in Elsasser variables}
\xi_x \cdot z^\pm = 0, \qquad M \xi_x + \xi_t z^+ = 0, \qquad M^T \xi_x + \xi_t z^- = 0.
\end{equation}

\subsection{The proof of Proposition \ref{Proposition for Hull theorem 1}}
Proposition \ref{Proposition for Hull theorem 1} gives our first estimation on the hull $K^{lc,\Lambda}_{r,s}$. Below, we further divide the proof of Proposition \ref{Proposition for Hull theorem 1} into five steps.

Let $0 \le \tau < 1$. Below, steps (i)--(v) are expressed under the assumption that $V \in \V_{\tau r,\tau s, \varepsilon_\tau}$, that is, $\abs{u+B} < \tau r + \varepsilon_\tau$, $\abs{u-B} < \tau s + \varepsilon_\tau$, $\abs{e \otimes e} < \varepsilon_\tau$, $\abs{S} < \varepsilon_\tau$, $\abs{B \times v} < \varepsilon_\tau$, $\abs{E} < \varepsilon_\tau$ and $\abs{\Pi} < \tau^2 rs + \varepsilon_\tau$. The constant $\varepsilon_\tau > 0$ varies from step to step.

\vspace{0.2cm}
(i) $V = (u, S_{u,B} + \Pi I, B, B \times u) \in K_{r,s}^{lc,\Lambda}$.

\vspace{0.2cm}
(ii) $V = (u, S_{u,B} + e \otimes e + \Pi I, B, B \times u) \in K_{r,s}^{lc,\Lambda}$.

\vspace{0.2cm}
(iii) $V = (u, S_{u,B} + S + \Pi I, B, B \times u) \in K_{r,s}^{lc,\Lambda}$.

\vspace{0.2cm}
(iv) $V = (u, S_{u,B} + S + \Pi I, B, B \times u + B \times v) \in K_{r,s}^{lc,\Lambda}$.

\vspace{0.2cm}
(v) $V = (u, S_{u,0} + S + \Pi I, 0, E) \in K_{r,s}^{lc,\Lambda}$.

\vspace{0.2cm}
\noindent Steps (i)--(v) are restated in Lemmas \ref{Lemma for step (i)}--\ref{Lemma for step (v)}.

In the first step we relax the constraints $\abs{z^+} = r$ and $\abs{z^-} = s$ to $\abs{z^+} \le r$ and $\abs{z^-} \le s$. 
The proof is most conveniently presented in Els\"{a}sser variables which facilates the search for $\Lambda$ combinations. For later use the statement is expressed in terms of the sets $\V_{\tau r,\tau s, \varepsilon_\tau}$.

\begin{lem}[ Relaxation of the normalisation] \label{Lemma for step (i)}
$(u,S_{u,B} + \Pi I,B,B \times u) \in K_{r,s}^{lc,\Lambda}$ whenever $\abs{z^+} < \tau r + \varepsilon_\tau^{(1)}$, $\abs{z^-} < \tau s + \varepsilon_\tau^{(1)}$ and $\abs{\Pi} < rs$, where $\varepsilon_\tau^{(1)} = \min\{r-\tau r, s - \tau s\}$.
\end{lem}

\begin{proof}
Suppose first $z^+,z^- \neq 0$. In terms of Els\"{a}sser variables,
\begin{align*}
   (z^+,z^-,z^+ \otimes z^- + \Pi I)
&= \lambda \left( \frac{r}{\abs{z^+}} z^+,z^-, \frac{r}{\abs{z^+}} z^+ \otimes z^- + \Pi I \right) \\
&+ (1-\lambda) \left( -\frac{r}{\abs{z^+}} z^+,z^-, -\frac{r}{\abs{z^+}} z^+ \otimes z^- + \Pi I \right)
\end{align*}
for $2\lambda -1 = |z^+|/r \in (0,1)$; here the $\Lambda$-direction is $(2 r z^+/|z^+|, 0, 2 r z^+/|z^+| \otimes z^-)$, so that \eqref{Wave cone conditions in Elsasser variables} are satisfied with any $\xi_x \in \{z^+,z^-\}^\perp \setminus \{0\}$ and $\xi_t = 0$.  Furthermore,
\begin{align*}
&\left( \pm \frac{r}{\abs{z^+}} z^+,z^-, \pm \frac{r}{\abs{z^+}} z^+ \otimes z^- + \Pi I \right) \\
= &\tilde{\lambda} \left( \pm \frac{r}{\abs{z^+}} z^+, \frac{s}{\abs{z^-}} z^-, \pm \frac{r}{\abs{z^+}} z^+ \otimes \frac{s}{\abs{z^-}} z^- + \Pi I \right) \\
+ &(1-\tilde{\lambda}) \left( \pm \frac{r}{\abs{z^+}} z^+, -\frac{s}{\abs{z^-}} z^-, \mp \frac{r}{\abs{z^+}} z^+ \otimes \frac{s}{\abs{z^-}} z^- + \Pi I \right) \in K_{r,s}^{1,\Lambda}
\end{align*}
for $2\tilde{\lambda}-1 = |z^-|/s \in (0,1)$; to show that the corresponding directions belong to $\Lambda$ we can again take $\xi_x \in \{z^+,z^-\}^\perp \setminus \{0\}$ and $\xi_t = 0$. Thus we have shown that $(z^+,z^-,z^+ \otimes z^- + \Pi I) \in K_{r,s}^{2,\Lambda}$.

Suppose next $z^+ = 0$ and $z^- \neq 0$. Now $(0,z^-,\pi I) = 2^{-1} (z^-,z^-,z^- \otimes z^- + \Pi I) + 2^{-1} (-z^-,z^-,-z^- \otimes z^- + \Pi I) \in K^{3,\Lambda}_{r,s}$, where we may choose $\xi$ with $\xi_t = 0$ and $\xi_x \in \{z^-\}^\perp \setminus \{0\}$. The remaining cases with $z^- = 0$ are similar. 
\end{proof}

Steps (ii)--(iii) are covered in the next two lemmas. This time we get rid of the constraint $S = S_{u,b} + \Pi I$. It is 
easier to deal first with a symmetric  rank-one matrix and then iterate. 

\begin{lem}[Adding a symmetric rank-one matrix] \label{Lemma for step (ii)}
$(u, S_{u,B} + e \otimes e + \Pi I, B, B \times u) \in K^{lc,\Lambda}_{r,s}$ whenever $\abs{z^+} < \tau r + \varepsilon_\tau^{(2)}$, $\abs{z^-} < \tau s + \varepsilon_\tau^{(2)}$, $\abs{e} < \varepsilon_\tau^{(2)}$ and $\abs{\Pi} < \tau^2 rs + \varepsilon_\tau^{(2)}$.
\end{lem}

\begin{proof}
We use the formula $S_{u,B} + e \otimes e = (S_{u+e,B} + S_{u-e,B})/2$ to write
\begin{align*}
  & (u, S_{u,B} + e \otimes e + \Pi I, B, B \times u) \\
= & \frac{1}{2} (u + e, S_{u+e,B} + \Pi I, B, B \times (u+e)) \\
+ & \frac{1}{2} (u - e, S_{u-e,B} + \Pi I, B, B \times (u-e)) \\
\eqdef & \frac{1}{2} (V_1 + V_2).
\end{align*}
Here Lemma \ref{Lemma for step (i)} gives $V_1, V_2 \in K^{lc,\Lambda}_{r,s}$ as long as $\varepsilon_\tau^{(2)} \le \varepsilon_\tau^{(1)}/2$. (We do not track such dependence of $\varepsilon^{(k)}_\tau$ on $\varepsilon^{(k-1)}_\tau$ explicitly in the forthcoming proofs.) The $\Lambda$-direction is
\[V_1-V_2 = (2e,2(u \otimes e + e \otimes u),0,2B \times e).\]
If $B \times e \neq 0$, we choose $\xi_x = B \times e$ and $\xi_t = - u \cdot B \times e$; if $B \times e = 0$, we choose any $\xi_x \in \{u,e\}^\perp \setminus \{0\}$ and $\xi_t = 0$.
\end{proof}

We then take further $\Lambda$-convex combinations to replace $e \otimes e$ by more general symmetric matrices.
\begin{lem} [Relaxation of the fluid side]\label{lemma for step (iii)}
$(u, S_{u,B} + S + \Pi I, B, B \times u) \in K_{r,s}^{lc,\Lambda}$ whenever $\abs{z^+} < \tau r + \varepsilon_\tau^{(3)}$, $\abs{z^-} < \tau s + \varepsilon_\tau^{(3)}$, $\abs{S} < \varepsilon_\tau^{(3)}$ and $\abs{\Pi} < \tau^2 rs + \varepsilon_\tau^{(3)}$.
\end{lem}

\begin{proof}
First we cover the case where $S = -e \otimes e$. Choose an orthogonal basis $\{e,f,g\}$ of $\R^3$, where $\abs{e} = \abs{f} = \abs{g}$. Write $I = \abs{e}^{-2} (e \otimes e + f \otimes f + g \otimes g)$ which, in combination with Lemma \ref{Lemma for step (ii)}, yields
\begin{align*}
   (u,S_{u,B} - e \otimes e + \Pi I, B, B \times u)
&= (u,S_{u,B} + f \otimes f + g \otimes g + (\Pi - \abs{e}^2) I, B, B \times u) \\
&= \frac{1}{2} (u,S_{u,B} + 2 f \otimes f + (\Pi - \abs{e}^2) I, B, B \times u) \\
&+ \frac{1}{2} (u,S_{u,B} + 2 g \otimes g + (\Pi - \abs{e}^2) I, B, B \times u) \\
&\in K_{r,s}^{lc,\Lambda};
\end{align*}
the $\Lambda$-direction is $\bar{V} = (0,2 f \otimes f - 2 g \otimes g, 0, 0)$ and we may choose $(\xi_x,\xi_t) = (e,0)$.

By noting that $(0, e \otimes e \pm f \otimes f, 0, 0) \in \Lambda$ for every $e,f \in \R^3$ and iterating, we obtain the case
\begin{equation} \label{Symmetric matrices via primitive metrics}
S = \sum_{i=1}^N c_i f_i \otimes f_i
\end{equation}
for any unit vectors $f_i \in \mathbb{S}^2$ and $c_i \in \R$ with $\sum_{i=1}^N \abs{c_i} < \varepsilon_\tau$. The proof is finished by noting that every $S \in \R_{\operatorname{sym}}^{3 \times 3}$ with $\abs{S} < \varepsilon_\tau$ can be written in the form \eqref{Symmetric matrices via primitive metrics}. Indeed, whenever $f$ and $g$ are unit vectors, we may write $f \otimes g + g \otimes f = 2^{-1} (f+g) \otimes (f+g) - 2^{-1} (f-g) \otimes (f-g)$.
\end{proof}

We have now covered the case where $V$ differs from an element of $K$ by the perturbation $S$ of the symmetric matrix part. Our next aim, in the following two lemmas, is to allow $E \neq B \times u$ in $V = (u,S,B,E)$.  Recall that
$K^{lc,\Lambda}_{r,s} \subset \mathscr{M}$. Thus, if $B \neq 0$, $E=B \times f, f \in \R^3$  is a necessary condition.  We will see next that it is 
also sufficient with the appropriate size normalisations.  We will make use of the formula
\begin{equation} \label{Formula for primitive metrics}
S_{u,B} = \frac{1}{2} (S_{u+\tilde u,B+\tilde B} - S_{\tilde u,\tilde B}) + \frac{1}{2} (S_{u-\tilde u,B-\tilde B} - S_{\tilde u,\tilde B}) \qquad (u,\tilde u,B,\tilde B \in \R^3).
\end{equation}
Finally, recall that in view of Proposition \ref{Proposition on rigidity of good Lamdba}, we are forced to use bad $\Lambda$-segments.
\begin{lem}[Relaxation of the magnetic side] \label{Relaxation of the magnetic side}
$(u, S_{u,B} + S + \Pi I, B, B \times u + B \times v) \in K_{r,s}^{lc,\Lambda}$ whenever $\abs{z^+} < \tau r + \varepsilon_\tau^{(4)}$, $\abs{z^-} < \tau s + \varepsilon_\tau^{(4)}$, $\abs{S} < \varepsilon_\tau^{(4)}$, $\abs{b \times v} < \varepsilon_\tau^{(4)}$ and $\abs{\Pi} < \tau^2 rs + \varepsilon_\tau^{(4)}$.
\end{lem}

\begin{proof}
We may assume that $B \times v \neq 0$ and $B \cdot v = 0$. Then $\abs{B \times v} = \abs{B} \abs{v}$. The difficulty is that if $B$ is very small, $v$ can be very large.

We denote $c \defeq (\abs{B}/\abs{v})^{1/2}$ so that 
\begin{equation}\label{small}\abs{c v} = |c^{-1} B| = \abs{B \times v}^{1/2} < |\varepsilon_\tau^{(4)}|^{1/2}.
\end{equation}
 We then use \eqref{Formula for primitive metrics} to show that $(u, B, S_{u,B} + S + \Pi I, B \times u + B \times v)$ is the middle point of a suitable  $\Lambda$ segment. Indeed,
\begin{align*}
  & (u,B,S_{u,B} + S + \Pi I, B \times (u+v)) \\
= & \frac{1}{2} (u+c v,S_{u+c v,(1+c^{-1}) B} - S_{cv,c^{-1}B} + S + \Pi I, B + c^{-1} B, (1+c^{-1}) B \times (u + cv)) \\
+ & \frac{1}{2} (u-c v,S_{u-c v,(1-c^{-1}) B} - S_{cv,c^{-1}B} + S + \Pi I, B - c^{-1} B, (1-c^{-1}) B \times (u - cv)).
\end{align*}
Notice that thanks to \eqref{small} we can apply Lemma~\ref{lemma for step (iii)} to deduce that the endpoints lie in $K_{r,s}^{lc,\Lambda}$. The direction of the segment is 
\[
\bar{V} = \left( 2 c v, 2 \left( u \otimes c v + c v \otimes u - \frac{2}{c} B \otimes B \right), 2 c^{-1} B, 2 \left( B \times c v + c^{-1} B \times u \right) \right),\]
which belongs to $\Lambda$ since  \eqref{Wave cone condition 4} is satisfied.
\end{proof}

The case $B=0$ needs to be dealt with separately, since lying in $\mathscr{M}$ does no longer constrain $E$.
The following lemma proves step (v) and completes the proof of Proposition \ref{Proposition for Hull theorem 1}.

\begin{lem}[The case $B=0$] \label{Lemma for step (v)}
$(u, 0, S_{u,0} + S + \Pi I, E) \in K_{r,s}^{lc,\Lambda}$ whenever $\abs{u} < \tau r + \varepsilon_\tau^{(5)}$, $\abs{u} < \tau s + \varepsilon_\tau^{(5)}$, $\abs{S} < \varepsilon_\tau^{(5)}$, $\abs{E} < \varepsilon_\tau^{(5)}$ and $\abs{\Pi} < \tau^2 rs + \varepsilon_\tau^{(5)}$.
\end{lem}

\begin{proof}
We choose orthogonal $e,f$ such that $E=e\times f$, $\abs{e} = \abs{f} = \abs{E}^{1/2} < (\varepsilon_\tau^{(5)})^{1/2}$. Using \eqref{Formula for primitive metrics}, we write
\begin{align*}
  & (u,S_{u,0} + S + \Pi I, 0, e \times f) \\
= & \frac{1}{2} (u + e \times f, S_{u+e \times f,e} - S_{e \times f, e}+ S + \Pi I, e, e \times (u + e \times f + f) \\
+ & \frac{1}{2} (u - e \times f, S_{u-e \times f, -e} - S_{e \times f, e} + S + \Pi I, - e, -e \times (u - e \times f + (2 e \times f - f)).
\end{align*}
By Lemma \ref{Relaxation of the magnetic side} the endpoints belong to $K_{r,s}^{lc,\Lambda}$. Now $\bar{V} = (2 e \times f, 2 (u \otimes e \times f + e \times f \otimes u), 2 e, 2 e \times (u + e \times f)) \in \Lambda$ since \eqref{Wave cone condition 4} is satisfied.
\end{proof}

\subsection{The proof of Proposition \ref{Proposition for Hull theorem 3}}
Recall that Proposition \ref{Proposition for Hull theorem 3} states the inclusion $\U_{r,s} \subset \cup_{\tau_0 < \tau < 1} K_{\tau r, \tau s}^{lc,\Lambda}$ and completes the proof of Theorem \ref{Hull theorem}.

\begin{proof}[Proof of Proposition \ref{Proposition for Hull theorem 3}]
Let $V \in \operatorname{int}_{\mathscr{M}}(K_{r,s}^{lc,\Lambda})$ and $0 < \tau_0 < 1$. By relative openness of $\operatorname{int}_{\mathscr{M}}(K_{r,s}^{lc,\Lambda})$, we may choose $\mu$ such that $\tau_0 < \sqrt{\mu} < 1$ and $V/\mu \in K_{r,s}^{lc,\Lambda}$. Now $V \in (\mu K_{r,s})^{lc,\Lambda}$ since the conditions $\bar{W} \in \Lambda$ and $\mu \bar{W} \in \Lambda$ are equivalent for all $\bar{W} \in \R^{15}$. It thus suffices to show that $\mu K_{r,s} \subset K_{\sqrt{ \mu} r, \sqrt{\mu} s}^{lc,\Lambda}$.

We use Els\"{a}sser variables. When $(\mu z^+, \mu z^-, \mu z^+ \otimes z^- + \mu \Pi I) \in \mu K_{r,s}$, we note that $\sqrt{\mu} \in (\mu,1)$ and write
\begin{align*}
     (\mu z^+, \mu z^-, \mu z^+ \otimes z^- + \Pi I)
&=   \lambda (\sqrt{\mu} z^+, \sqrt{\mu} z^-, \mu z^+ \otimes z^- + \Pi I) \\
&+   (1-\lambda) (-\sqrt{\mu} z^+, -\sqrt{\mu} z^-, \mu z^+ \otimes z^- + \Pi I) \\
&\in K_{\sqrt{ \mu} r, \sqrt{\mu} s}^{1,\Lambda}
\end{align*}
for $2\lambda-1 = \sqrt{\mu} \in (0,1)$; here $\bar{V} = (2 \sqrt{\mu} z^+, 2 \sqrt{\mu} z^-, 0) \in \Lambda$. Hence $\mu K_{r,s} \subset K_{\sqrt\mu r, \sqrt\mu s}^{1,\Lambda}$.
\end{proof}

\section{The proof of Theorem \ref{MHD 3D theorem}} \label{The proof of MHD 3D theorem}
This chapter is dedicated to proving Theorem \ref{MHD 3D theorem}. In \textsection \ref{Restricted subsolutions} we define the set of subsolutions that we use in the proof, and the main steps of the proof are listed in \textsection \ref{The main lemmas}. The proof itself is carried out in the rest of the chapter.

\subsection{Restricted subsolutions} \label{Restricted subsolutions}
We intend to prove Theorem \ref{MHD 3D theorem} by using subsolutions that take values in $\U_{r,s}$ and whose $B$ and $E$ components arise via $P_B$ and $P_E$. For this, recall the notations
\[W = (u,S,d\varphi,d\psi), \qquad V = p(W) = (u,S,d\varphi \wedge d \psi).\]
Fix a non-empty bounded domain $\Omega \subset \R^3 \times \R$, and let $r,s > 0$, $r \neq s$.

\begin{defin}
The set of \emph{restricted subsolutions} is defined as
\begin{align*}
        X_0
\defeq & \{V = (u,S,\omega) \in C_c^\infty(\R^4,\R^{15}) \colon \text{there exists } \varphi, \psi \in C_c^\infty(\R^4) \text{ such that } \\
& \omega = d\varphi \wedge d \psi, \, \mathcal{L}(V) = 0, \; \supp(u,S,\varphi,\psi) \subset \Omega \text{ and } V(x,t) \in \U_{r,s} \forall (x,t) \in \R^4\}.
\end{align*}
We denote by $X$ the weak sequential closure of $X_0$ in $L^2(\R^4; \overline{\operatorname{co}}(K_{r,s}))$.
\end{defin}

Now $X \ni \{0\}$ is a compact metrisable space, and we denote a metric by $d_X$.

\subsection{The main steps of the proof} \label{The main lemmas}
Following~\cite{DLS09}, our main aim is to prove Proposition \ref{Main proposition} below. Once Proposition \ref{Main proposition} is proved, Theorem \ref{MHD 3D theorem} follows rather easily in \textsection \ref{Completion of the proof of MHD 3D theorem}.

\begin{prop} \label{Main proposition}
There exists $C = C_{r,s} > 0$ with the following property. If $V = (u,S,d \varphi \wedge d \psi) \in X_0$, then there exist $V_\ell = (u_\ell, S_\ell, d \varphi_\ell \wedge d \psi_\ell) \in X_0$ such that $d_X(V_\ell,V_0) \to 0$ and
\begin{align*}
&\int_\Omega (\abs{u_\ell(x,t)}^2 + \abs{B_\ell(x,t)}^2 - \abs{u(x,t)}^2 - \abs{B(x,t)}^2) \, dx \, dt \\
\ge \; & C \int_\Omega \left( \frac{r^2+s^2}{2} - \abs{u(x,t)}^2 - \abs{B(x,t)}^2 \right) \, dx \, dt.
\end{align*}
\end{prop} 

For the proof of Proposition \ref{Main proposition} we need a so-called \emph{perturbation property}, formulated in our setting in Proposition \ref{Proposition on approximation of laminates}. To motivate the formulation of Proposition \ref{Proposition on approximation of laminates}, we note that Theorem \ref{Hull theorem} implies the following proposition where we choose any $\varepsilon_\tau = \varepsilon_{\tau,r,s}$ such that
\[\mathscr{O}_\tau \defeq B_{\mathscr{M}}(K_{\tau r, \tau s},\varepsilon_\tau) \subset \U_{r,s};\]
recall from Theorem \ref{Hull theorem} that for any $\tau_0 \in (0,1)$ we have
\[\U_{r,s} = \bigcup_{\tau_0 < \tau < 1} \mathscr{O}_\tau^{lc,\Lambda_g}.\]

\begin{prop} \label{Proposition on existence of finite-order good laminates}
Let $V_0 \in \U_{r,s}$. Then for every large enough $\tau \in (0,1)$ there exists
\[\nu = \sum_{\mathbf{j} \in \{1,2\}^N} \mu_\mathbf{j} \delta_{V_\mathbf{j}} \in \mathcal{L}_g(\mathscr{O}_\tau)\]
with barycentre $\bar{\nu} = V_0$ and $[V_{\mathbf{j}',1},V_{\mathbf{j}',2}] \subset \U_{r,s}$ for all $\mathbf{j}' \in \{1,2\}^k$, $1 \le k \le N-1$. Furthermore, for each $V_\mathbf{j} = (u_\mathbf{j},S_\mathbf{j},v_\mathbf{j} \wedge w_\mathbf{j})$, $\mathbf{j} \in \{1,2\}^N$, we have
\begin{equation} \label{Estimate on distance from K}
\frac{r^2 + s^2}{2} - \abs{u_0}^2 - \abs{B_0}^2 \le 2 (\abs{u_\mathbf{j}}^2 + \abs{B_\mathbf{j}}^2 - \abs{u_0}^2 - \abs{B_0}^2).
\end{equation}
\end{prop}

Indeed, if $V = (u,S,B,E) \in K_{r,s}$, then $\abs{u}^2 + \abs{B}^2 = (r^2 + s^2)/2$ whereas, since $\supp(\nu) \subset \mathscr{O}_\tau$, $\abs{u_\mathbf{j}}^2 + \abs{B_\mathbf{j}}^2 \geq \tau^2 (r^2+s^2)/2 - \varepsilon_\tau$. Therefore, \eqref{Estimate on distance from K} follows by choosing $\tau \in (0,1)$ large enough.

Whereas Proposition \ref{Proposition on existence of potentials} says, roughly speaking, that every good $\Lambda$-segment can be approximated by oscillating mappings with certain properties, Proposition \ref{Proposition on approximation of laminates} makes an analogous claim about good laminates.

\begin{prop} \label{Proposition on approximation of laminates}
Let $Q \subset \R^4$ be a cube, and let $V_0 = p(W_0) \in \U_{r,s}$. If $\omega_0 = v_0 \wedge w_0 = 0$, then assume that $v_0 = w_0 = 0$. Choose $\nu \in \mathcal{L}_g(\mathscr{O}_\tau)$ with $\bar{\nu} = V_0$ via Proposition \ref{Proposition on existence of finite-order good laminates}.

For every $\varepsilon > 0$ there exist
\[W_\ell \defeq W_0 + (\bar{u}_\ell,\bar{S}_\ell,d \bar{\varphi}_\ell, d \bar{\psi}_\ell) \in W_0 + C_c^\infty(Q;\R^{17})\]
with the following properties:
\begin{enumerate}[\upshape (i)]
\item $\mathcal{L}(V_\ell) = 0$ and $V_\ell(x,t) \in \U_{r,s}$ for all $(x,t) \in \Omega$.

\item There exist pairwise disjoint open subsets $A_\mathbf{j} \subset Q$ with $\abs{\abs{A_\mathbf{j}}-\mu_\mathbf{j}} < \varepsilon$ such that
\[V_\ell(x,t) = V_\mathbf{j} \; \text{ for all } \mathbf{j} \in \{1,2\}^N \text{ and } (x,t) \in A_\mathbf{j}\]
and $d\bar{\varphi}_\ell$ and $d\bar{\psi}_\ell$ are locally constant in $A_\mathbf{j}$.

\item For every $(x,t) \in Q$ there exist $\mathbf{j}' \in \{1,2\}^k$ and $\tilde{W} = \tilde{W}(x,t) \in \R^{17}$ such that
\[p(\tilde{W}) \in [V_{\mathbf{j}',1}, V_{\mathbf{j}',2}], \qquad
 |W_\ell(x,t) - \tilde{W}| < \varepsilon, \qquad |V_\ell(x,t) - p(\tilde{W})| < \varepsilon.\]

\item $V_\ell - V_0 \rightharpoonup 0$ in $L^2(Q;\R^{15})$.
\end{enumerate}
\end{prop}

Condition (iii) says, in particular, that at every $(x,t) \in \Omega$, $V_\ell(x,t)$ is close to one of the $\Lambda_g$-segments $[V_{\mathbf{j}',1},V_{\mathbf{j}',2}]$, where $\mathbf{j}' \in \{1,2\}^k$, $1 \le k \le N-1$. We will also need the estimate on $W_\ell(x,t)$. Proposition \ref{Proposition on approximation of laminates} is proved by a standard induction via Proposition \ref{Proposition on existence of potentials}; we sketch the main ideas.

\begin{proof}
The proof follows by iteratively modifying the sequence at the sets where it is locally affine (via Proposition \ref{Proposition on existence of potentials}) and using a diagonal argument. Namely, if $N=1$, the result follows from Proposition \ref{Proposition on existence of potentials}.  Suppose now that $1 \le k \le N-1$ and that we have costructed a sequence of mappings $W_{\ell_k} = (u_{\ell_k}, S_{\ell_k}, d\varphi_{\ell_k}, d\psi_{\ell_k})$ which satisfies (i), (iii) and (iv) and furthermore (ii) holds with the condition $\mathbf{j} \in \{1,2\}^N$ replaced by $\mathbf{j} \in \{1,2\}^k$.

Fix $\mathbf{j} \in \{1,2\}^k$. We cover $A_\mathbf{j}$ by disjoint cubes up to a set of small measure and modify $W_{\ell_k}$ in each cube via Proposition \ref{Proposition on existence of potentials}. This gives rise to a new sequence which we again modify at the sets $A_\mathbf{j}$, $\mathbf{j} \in \{1,2\}^{k+1}$, where it is locally affine. Note that we can use Proposition \ref{Proposition on existence of potentials} iteratively because in each $A_{\mathbf{j}}$, claim (iii) of Proposition \ref{Proposition on existence of potentials} implies that either $d\varphi_{k_\ell} \wedge d\psi_{k_\ell} \neq 0$ or $d\varphi_{k_\ell} = d\psi_{k_\ell} = 0$.  Finally a standard diagonal argument provides the norm bounds.
\end{proof}

\subsection{Modifications at the set where $d\varphi(x,t) \wedge d\psi(x,t) = 0$ }
The following issue needs to be addressed in the proof of Proposition \ref{Main proposition}: on one hand, the mapping $V = (u,S,d\varphi \wedge d\psi) \in X_0$ can have a large set where $d\varphi(x,t) \wedge d\psi(x,t) = 0$ but $(d\varphi(x,t),d\psi(x,t)) \neq 0$, and on the other hand, in Proposition \ref{Proposition on existence of potentials}, in the case $\omega_0 = v_0 \wedge w_0 = 0$, we only constructed potentials when $v_0 = w_0 = 0$. We therefore modify $W$ around points $(x,t) \in \Omega$ where $d\varphi(x,t) \wedge d\psi(x,t) = 0$ but $(d\varphi(x,t),d\psi(x,t)) \neq 0$ making $W$ look essentially constant there. 

\begin{lem} \label{Lemma on modifying elements of X0}
Suppose $V \in X_0$, and let $\varepsilon > 0$. Then there exists $\tilde{V} = (\tilde{u},\tilde{S},d\tilde{\varphi} \wedge d\tilde{\psi}) \in X_0$ such that $\|V-\tilde{V}\|_{L^\infty} < \varepsilon$ and
\begin{align*}
& |\{(x,t) \in \Omega \colon d\tilde{\varphi}(x,t) \wedge d\tilde{\psi}(x,t) \neq 0\} \cup \operatorname{int}(\{(x,t) \in \Omega \colon d\tilde{\varphi}(x,t) = d\tilde{\psi}(x,t) = 0\})| \\
> & (1-\varepsilon) \abs{\Omega}.
\end{align*}
\end{lem}

\begin{proof}
Assume, without loss of generality, that $0 < \varepsilon < \min_\Omega \text{dist}(V,\partial\U_{r,s})$. Then the inequality $\|V-\tilde{V}\|_{L^\infty} < \varepsilon$ ensures that $\tilde{V}$ takes values in $\U_{r,s}$.

Since $W$ is absolutely continuous, we cover $\Omega$ by all the cubes $Q_i \subset \Omega$ with centers $(x_i,t_i)$ and the following properties:
\begin{itemize}
\item If $d\varphi(x_i,t_i) \wedge d\psi(x_i,t_i) \neq 0$, then $d\varphi \wedge d\psi \neq 0$ in $Q_i$.

\item If $d\varphi(x_i,t_i) \wedge d\psi(x_i,t_i) = 0$, then we have $\sup_{(x,t) \in Q_i} \abs{W(x,t) - W(x_i,t_i)} < \varepsilon^2/[C (\norm{W}_{L^\infty} + 1)]$.
\end{itemize}
\noindent Such cubes exist for every $(x_i,t_i) \in \Omega$, and therefore they form a Vitali cover of $\Omega$. By the Vitali Covering Theorem, we may choose a finite, pairwise disjoint subcollection $\{Q_1,\ldots,Q_N\}$ with $|\Omega \setminus \cup_{i=1}^N Q_i| < \varepsilon/2$.

We intend to modify $V$ in each $Q_i$ where $d\varphi(x_i,t_i) \wedge d\psi(x_i,t_i) = 0$. Fix such $Q_i$, and let $R \subset Q_i$ be a subcube with center $(x_i,t_i)$ and $\abs{R} = (1-\varepsilon/2) \abs{Q_i}$. Choose $\delta > 0$ such that $(1-\delta)^4 = 1-\varepsilon/2$; now $l(r) = (1-\delta) l(Q_i)$. Choose a smooth cutoff function $\chi_R$ with $\chi_R|_R = 1$ and $\abs{\nabla \chi_R} \le C /[\delta l(Q_i)]$. Define $g \in C^\infty(Q;Q)$ by
\[g(x,t) \defeq (x,t) + \chi_R(x,t) [(x_i,t_i)-(x,t)]\]
so that $g(x,t) = (x_i,t_i)$ is constant in $R$ and $g = \text{id}$ near $\partial Q_i$. Set $\tilde{\varphi} \defeq \varphi \circ g$ and $\tilde{\psi} \defeq \psi \circ g$ so that
\begin{equation} \label{Gradients of modified potentials}
\nabla_{x,t} \, \tilde{\varphi} = D^T_{x,t} \, g \nabla_{x,t} \, \varphi \circ g, \qquad \nabla_{x,t} \, \tilde{\psi} = D^T_{x,t} \, g \nabla_{x,t} \, \psi \circ g.
\end{equation}
Thus $|\{(x,t) \in Q_i \colon d\tilde{\varphi}(x,t) = d\tilde{\psi}(x,t) = 0\}| \ge (1-\varepsilon/2) \abs{Q_i}$.

The claim will be proved once we show that $\|d\tilde{\varphi} \wedge d\tilde{\psi}\| < \varepsilon/2$ in $Q_i$; then $\|V-\tilde{V}\|_{L^\infty} < \varepsilon$. To this end, we fix $(x,t) \in Q_i$ and estimate
\[\abs{D g(x,t)} = \abs{(1-\chi_R(x,t)) I + [(x_i,t_i)-(x,t)] \otimes \nabla \chi_R(x,t)} \le \frac{C}{\delta}\]
and
\begin{align*}
\abs{d\varphi(g(x,t)) \wedge d\psi(g(x,t))}
&\le \abs{d\varphi(g(x,t)) \wedge (d\psi(g(x,t)) - d\psi(g(x_i,t_i))} \\
&+ \abs{(d\varphi(x,t)) - d\varphi(x_i,t_i)) \wedge d\psi(x_i,t_i)} \\
&\le C' \norm{W}_{L^\infty} \abs{W(x,t) - W(x_i,t_i)} < \varepsilon^2/C.
\end{align*}
Now $\delta = 1 - (1-\varepsilon/2)^{1/4}$ implies that $\varepsilon/2 = 1 - (1-\varepsilon/2) = \delta (1 + (1-\varepsilon/2)^{1/4} + (1-\varepsilon/2)^{2/4} + (1-\varepsilon/2)^{3/4}) < 4\delta$. Thus, whenever $\abs{v_1} = \abs{v_2} = 1$, we have
\begin{align*}
  & |[d\tilde{\varphi}(x,t) \wedge d\tilde{\psi}(x,t)] (v_1,v_2)| \\
= & |[d\varphi(g(x,t)) \wedge d\psi(g(x,t))] (D^T g(x,t) v_1, D^T g(x,t) v_2)| < \frac{\varepsilon}{2}.
\end{align*}
\end{proof}

By Lemma \ref{Lemma on modifying elements of X0} and a standard diagonal argument, it suffices to prove Proposition \ref{Main proposition} for every $\varepsilon > 0$ and every 
mapping $V \in X_0$ such that
\[\tilde{\Omega} \defeq \{(x,t) \in \Omega \colon d\varphi(x,t) \wedge d\psi(x,t) \neq 0\} \cup \operatorname{int}(\{(x,t) \in \Omega \colon d\varphi(x,t) = d\psi(x,t) = 0\})\]
satisfies
\begin{align}
& |\tilde{\Omega}| > (1-\varepsilon) \abs{\Omega}, \label{Assumption on tilde Omega 1} \\
& \begin{aligned}
& \int_\Omega \left( \frac{r^2+s^2}{2} - \abs{u(x,t)}^2 - \abs{B(x,t)}^2 \right) \, dx \, dt \label{Assumption on tilde Omega 2} \\
& \le \, 2 \int_{\tilde{\Omega}} \left( \frac{r^2+s^2}{2} - \abs{u(x,t)}^2 - \abs{B(x,t)}^2 \right) \, dx \, dt. \end{aligned}
\end{align}

\subsection{Proof of Proposition \ref{Main proposition}}
Assuming that \eqref{Assumption on tilde Omega 1}--\eqref{Assumption on tilde Omega 2} hold, we wish to construct the mappings $V_\ell = p(W_\ell)$ of Proposition \ref{Main proposition} by suitably modifying a discretisation argument from~\cite{DLS09}. Given $V = p(W) \in X_0$ we cover $\tilde{\Omega}$, up to a small set, by cubes $Q_i \subset \tilde{\Omega}$ with center $(x_i,t_i)$ such that $W$ varies very little in $Q_i$. We then approximate $W$ by $W(x_i,t_i)$ in each $Q_i$. Now $V(x_i,t_i) = \bar{\nu}$ for some $\nu = \sum_{\mathbf{j} \in \{1,2\}^N} \mu_\mathbf{j} \delta_{V_\mathbf{j}} \in \mathcal{L}_g(\mathscr{O}_\tau)$ with $\tau$ close to $1$, and we set $W_\ell \defeq [W - W(x_i,t_i)] + [W(x_i,t_i) + (\bar{u}_\ell,\bar{S}_\ell,d\bar{\varphi}_\ell, d\bar{\psi}_\ell)]$, where $(\bar{u}_\ell,\bar{S}_\ell,d\bar{\varphi}_\ell, d\bar{\psi}_\ell)$ is given by Proposition \ref{Proposition on approximation of laminates}.

On one hand, the discretisation needs to be fine enough that $V_\ell = p(W_\ell)$ does not take values outside $\U_{r,s}$, and on the other hand, the cubes need to cover a substantial proportion of $\tilde{\Omega}$. Both properties are ensured by the following application of the Vitali Covering Theorem.
 
\begin{lem} \label{Discretisation lemma}
Suppose $\varepsilon > 0$ and $V = p(W) \in X_0$ satisfies \eqref{Assumption on tilde Omega 1}--\eqref{Assumption on tilde Omega 2}. Let $\gamma > 0$. Then there exist pairwise disjoint cubes $Q_i \subset \tilde{\Omega}$ with centers $(x_i,t_i)$ and parameters $\delta_i > 0$ with the following properties:
\begin{enumerate}[\upshape (i)]
\item For every $i \in \N$, there exists $\tau_i \in (0,1)$ and $\nu = \sum_{\mathbf{j} \in \{1,2\}^N} \mu_\mathbf{j} \delta_{V_\mathbf{j}} \in \mathcal{L}_g(\mathscr{O}_{\tau_i})$ with barycentre $\bar{\nu} = V(x_i,t_i)$, where $\nu$ given by Proposition \ref{Proposition on existence of finite-order good laminates}.

\item $B_{\mathscr{M}}([V_{\mathbf{j}',1},V_{\mathbf{j}',2}]),\delta_i) \subset \U_{r,s}$ for all $i \in \N$ and $\mathbf{j}' \in \{1,2\}^k$, $1 \le k \le N-1$.

\item $\sup_{(x,t) \in Q_i} \abs{W(x,t)-W(x_i,t_i)} < \gamma \delta_i$.

\item $|\tilde{\Omega} \setminus \cup_{i=1}^\infty Q_i| = 0$.
\end{enumerate}
\end{lem}

\begin{proof}
Let $(x,t) \in \tilde{\Omega}$. Since $V(x,t) \in \U_{r,s}$, by Theorem \ref{Hull theorem} there exist $\tau \in (0,1)$ and $\nu = \sum_{\mathbf{j} \in \{1,2\}^N} \mu_\mathbf{j} \delta_\mathbf{j} \in \mathcal{L}_g(\mathscr{O}_\tau)$ with barycentre $\bar{\nu} = V(x,t)$. Furthermore, there exists $\delta > 0$ such that $B_{\mathscr{M}}([V_{\mathbf{j}',1},V_{\mathbf{j}',2}]),\delta) \subset \U_{r,s}$ whenever $\mathbf{j}' \in \{1,2\}^k$, $1 \le k \le N-1$. Since $W$ is continuous, there exists a cube $Q \subset \tilde{\Omega}$ with center $(x,t)$ such that $\sup_{(x',t') \in Q} \abs{W(x',t') - W(x,t)} < \gamma \delta_i$.

The collection of cubes chosen above forms a Vitali cover of $\tilde{\Omega}$, and therefore, by the Vitali Covering Theorem, there exists a countable, pairwise disjoint subcollection $\{Q_i\}_{i \in \N}$ with $|\tilde{\Omega} \setminus \cup_{i=1}^\infty Q_i| = 0$.
\end{proof}

\begin{proof}[Proof of Proposition \ref{Main proposition}]
Let $\varepsilon > 0$ and $V \in X_0$, and suppose $V$ satisfies \eqref{Assumption on tilde Omega 1}--\eqref{Assumption on tilde Omega 2}. Let $0 < \gamma = \gamma_V \ll [\min_{(x,t) \in \tilde{\Omega}} ((r^2 + s^2)/2 - \abs{u(x,t)}^2 - \abs{B(x,t)}^2)]^{1/2}$ (to be determined later) and choose cubes $Q_i \subset \tilde{\Omega}$ via Lemma \ref{Discretisation lemma}. At each $Q_i$ define
\[W_\ell \defeq W + (\bar{u}_\ell,\bar{S}_\ell,d \bar{\varphi}_\ell, d \bar{\psi}_\ell) \in C^\infty(Q_i;\R^{17}),\]
where $(\bar{u}_\ell,\bar{S}_\ell,d \bar{\varphi}_\ell, d \bar{\psi}_\ell) \in C_c^\infty(Q_i;\R^{17})$ is given by Proposition \ref{Proposition on approximation of laminates}.

We now intend to show that
\[V_\ell = (u + \bar{u}_\ell, S + \bar{S}_\ell, (d\varphi + d\bar{\varphi}_\ell) \wedge (d\psi + d\bar{\psi}_\ell)) \eqdef (u_\ell,S_\ell,d\varphi_\ell \wedge d\psi_\ell)\]
takes values in $\U_{r,s}$ for every $\ell \in \N$; then $V_\ell \in X_0$ by construction.

Fix a cube $Q_i$ and write $V(x_i,t_i) = p(W(x_i,t_i)) \in \U_{r,s}$ as a barycentre of $\nu = \sum_{\mathbf{j} \in \{1,2\}^N} \mu_\mathbf{j} \delta_{V_\mathbf{j}} \in \mathcal{L}_g(\mathscr{O}_\tau)$. By Lemma \ref{Discretisation lemma}, whenever $\mathbf{j}' \in \{1,2\}^k$ with $1 \le k \le N-1$, we have $B_{\mathscr{M}}([V_{\mathbf{j}',1},V_{\mathbf{j}',2}],\delta_i) \subset \U_{r,s}$. 

Let $(x,t) \in Q_i$. By Lemma \ref{Discretisation lemma}, $\abs{W(x,t) - W(x_i,t_i)} < \gamma \delta_i$. By Proposition \ref{Proposition on approximation of laminates}, there exists $\tilde{W}$ such that $p(\tilde{W}) \in [V_{\mathbf{j}',1},V_{\mathbf{j}',2}]$ for some $\mathbf{j}' \in \{1,2\}^k$ and some $k \le N_1$, and $|W(x_i,t_i) + (\bar{u}_\ell,\bar{S}_\ell,d \bar{\varphi}_\ell, d \bar{\psi}_\ell)(x,t) - \tilde{W}| < \gamma \delta_i$. Thus $|W_\ell(x,t)-\tilde{W}| < 2 \gamma \delta_i$. Hence, whenever $\gamma > 0$ is small enough (independently of $i$), we conclude that $|V_\ell(x,t) -p(\tilde{W})| < \delta_i$ and $V_\ell(x,t) \in \U_{r,s}$.

\vspace{0.3cm}
Whenever $\gamma^2 < \min_{(x,t) \in \tilde{\Omega}} [(r^2 + s^2)/2 - \abs{u(x,t)}^2 - \abs{B(x,t)}^2]/3$, condition (iii) of Lemma \ref{Discretisation lemma} and the property $|\tilde{\Omega} \setminus \cup_{i \in \N} Q_i| = 0$ yield a finite subcollection of cubes such that
\begin{align*}
& \int_{\tilde{\Omega}} \left( \frac{r^2+s^2}{2} - \abs{u(x,t)}^2 - \abs{B(x,t)}^2 \right) dx \, dt \\
\le &2 \sum_{i=1}^M \left( \frac{r^2 + s^2}{2} - \abs{u(x_i,t_i)}^2 - \abs{B(x_i,t_i)}^2 \right) \abs{Q_i}.
\end{align*}
Let $1 \le i \le M$ and write $V(x_i,t_i) = \bar{\nu}$, where $\nu = \sum_{\mathbf{j} \in \{1,2\}^N} \mu_\mathbf{j} \delta_{V_\mathbf{j}}$ is given by Proposition \ref{Proposition on existence of finite-order good laminates}. In particular, $V_\ell(x,t) = V_\mathbf{j}$ in each $A_\mathbf{j} \subset Q_i$. Now Proposition \ref{Proposition on existence of finite-order good laminates} and Lemma \ref{Discretisation lemma} give
\begin{align*}
& \left( \frac{r^2 + s^2}{2} - \abs{u(x_i,t_i)}^2 - \abs{B(x_i,t_i)}^2 \right) \abs{Q_i} \\
\le & 3 \sum_{\mathbf{j} \in \{1,2\}^N} (\abs{u_\mathbf{j}}^2 + \abs{B_\mathbf{j}}^2 - \abs{u(x_i,t_i)}^2 - \abs{B(x_i,t_i)}^2) \abs{A_\mathbf{j}} \\
\le & 4 \sum_{\mathbf{j} \in \{1,2\}^N} \int_{A_\mathbf{j}} (\abs{u_\ell(x,t)}^2 + \abs{B_\ell(x,t)}^2 - \abs{u(x,t)}^2 - \abs{B(x,t)}^2) \, dx \, dt.
\end{align*}
Indeed, since $V_\ell = V_\mathbf{j}$ in each $A_\mathbf{j}$, the last estimate is equivalent to the inequality
\begin{align*}
& 4 \sum_{\mathbf{j} \in \{1,2\}^N} \int_{A_\mathbf{j}} (\abs{u(x,t)}^2 + \abs{B(x,t)}^2 - \abs{u(x_i,t_i)}^2 - \abs{B(x_i,t_i)}^2) \, dx \, dt \\
\le & \sum_{\mathbf{j} \in \{1,2\}^N} \abs{A_\mathbf{j}} [\abs{u_\mathbf{j}}^2 + \abs{B_\mathbf{j}}^2 - \abs{u(x_i,t_i)}^2 - \abs{B(x_i,t_i)}^2],
\end{align*}
which in turn is ensured if $\tau$ is large enough and $\gamma^2 \ll \min_{(x,t) \in \tilde{\Omega}} [(r^2 + s^2)/2 - \abs{u(x,t)}^2 - \abs{B(x,t)}^2]$ is small enough. This proves the claim.
\end{proof}

\subsection{Completion of the proof of Theorem \ref{MHD 3D theorem}} \label{Completion of the proof of MHD 3D theorem}
\begin{proof}[Proof of Theorem \ref{MHD 3D theorem}]
The functional $V \mapsto \int_{\R^4} \abs{V(x,t)}^2 \, dx \, dt$ is a Baire-1 map in $X$, and thus its poins of continuity are residual in $X$ (see~\cite[Lemma 4.5]{DLS09}). Let now $V \in X$ be a point of continuity and choose a sequence of mappings $\tilde{V}_\ell \in X_0$ with $d(\tilde{V}_\ell,V) \to 0$. By Proposition \ref{Main proposition} and a standard diagonal argument, we find $V_\ell \in X_0$ with $d(V_\ell,V) \to 0$ and
\begin{align}
& \liminf_{\ell \to \infty} \int_{\Omega} (\abs{u_\ell(x,t)}^2 + \abs{B_\ell(x,t)}^2 - \abs{u(x,t)}^2 - \abs{B(x,t)}^2) \, dx \, dt \label{Limit estimate 1} \\
\ge \, & C \int_{\Omega} \left( \frac{r^2 + s^2}{2} - |u(x,t)|^2 - |B(x,t)|^2 \right) \, dx \, dt. \label{Limit estimate 2}
\end{align}
Since $V \mapsto \norm{V}_{L^2}^2$ is continuous at $V$ and $d(V_\ell,V) \to 0$, it follows that $\norm{V_\ell}_{L^2}^2 \to \norm{V}_{L^2}^2$. Thus $\norm{V_\ell-V}_{L^2} \to 0$ which, combined with \eqref{Limit estimate 1}--\eqref{Limit estimate 2}, gives $\abs{u(x,t)}^2 + \abs{B(x,t)}^2 = r^2 + s^2$ a.e. $(x,t) \in \Omega$. Thus
\[V(x,t) \in \overline{\U_{r,s}} \cap \{(u,S,B,E) \colon \abs{u+B} = r, \; \abs{u-B} = s\} \subset K_{r,s}\]
a.e. $(x,t) \in \Omega$. On the other hand, by the definition of $X$ we have $V(x,t) = 0$ a.e. $(x,t) \in (\R^3 \times \R) \setminus \Omega$. Now $V$ has all the sought properties.
\end{proof}

\begin{appendix}

\section{The ill-definedness of magnetic helicity and mean-square magnetic potential in the whole space} \label{The ill-definedness of magnetic helicity and mean-square magnetic potential in the whole space}

We state two simple results which indicate that mean-square magnetic potential and magnetic helicity are \emph{not} well-defined quantities for $L^2$-integrable solutions of ideal MHD in $\R^2$ or $\R^3$. We denote $L^2_\sigma(\R^n;\R^n) \defeq \{v \in L^2(\R^n;\R^n) \colon \nabla \cdot v = 0\}$ when $n \in \{2,3\}$.

\begin{prop} \label{Baire category proposition}
There exists $v \in L^2_\sigma(\R^2;\R^2)$ with the following property: if $\Psi \in \mathcal{D}'(\R^2)$ satisfies $\nabla^\perp \Psi = v$, then $\Psi \notin L^2(\R^2)$.
\end{prop}

Proposition \ref{Baire category proposition} is proved by choosing $\Theta \in \dot{W}^{1,2}(\R^2)$ such that $\Theta + C \notin W^{1,2}(\R^2)$ for every $C \in \R$ and setting $v \defeq \nabla^\perp \Theta$. The 3D result requires somewhat more work. Here we choose a smooth $v$ in order to make $\Psi \cdot v$ well-defined for all $\Psi \in \mathcal{D}'(\R^3,\R^3)$.

\begin{prop}
There exists $v \in L^2_\sigma(\R^3,\R^3) \cap C^\infty(\R^3,\R^3)$ with the following property: whenever $\Psi \in \mathcal{D}'(\R^3,\R^3)$ satisfies $\nabla \times \Psi = v$, we have $\Psi \cdot v \not\in L^1(\R^3)$.
\end{prop}

\begin{proof}
Fix $\psi_0 \in C_c^\infty(B(0,1),\R^3)$ such that $\psi_0$ and $\varphi_0 \defeq \nabla \times \psi_0$ satisfy $\int_{B(0,1)} \psi_0(x) \cdot \varphi_0(x) \, dx \neq 0$. (Choose, e.g., $\psi_0(x) = \chi(x) (1,x_3,0)$, where $\chi \in C_c^\infty(B(0,1)$ with $\chi(0) > 0$.) Set $\psi_0(x) = \varphi_0(x) = 0$ outside $B(0,1)$.

Fix points $x_j \in \R^3$ and radii $R_j > 0$ such that the balls $B(x_j,R_j)$, $j \in \N$, are mutually disjoint. For every $j \in \N$ denote
\[\psi_j(x) \defeq \frac{\psi_0 \left( \frac{x-x_j}{R_j} \right)}{R_j^{1/2}}, \qquad \varphi_j(x) \defeq \frac{\varphi_0 \left( \frac{x-x_j}{R_j} \right)}{R_j^{3/2}},\]
so that $\text{supp}(\psi_j) \subset B(x_j,R_j)$, $\nabla \times \psi_j = \varphi_j$ and $\norm{\varphi_j}_{L^2} = \norm{\varphi_0}_{L^2}$. Define $v \in L^2_\sigma(\R^3,\R^3)$ by
\[v(x) \defeq \sum_{j=1}^\infty \frac{1}{j^2} \varphi_j(x).\]

Suppose now $\Psi \in \mathcal{D}'(\R^3,\R^3)$ satisfies $\nabla \times \Psi = v$ and $\Psi \cdot v \in L^1(\R^3)$. Given $j \in \N$, note that in $B(x_j,R_j)$ we have $\Psi = \psi_j + \nabla g_j$, where $g_j \in \mathcal{D}'(B(x_j,t_j))$. Thus, by using the fact that $\nabla g_j \cdot \varphi_j = 0$ for every $j \in \N$ we get
\begin{align*}
\int_{\R^3} \abs{\Psi(x) \cdot v(x)} \, dx
&\ge \sum_{j=1}^\infty \frac{1}{j^2} \left| \int_{B(x_j,R_j)} \psi_j(x) \cdot \varphi_j(x) \, dx \right| \\
&\ge \sum_{j=1}^\infty \frac{1}{j^2} R_j \left| \int_{B(0,1)} \psi_0(x) \cdot \varphi_0(x) \, dx \right|,
\end{align*}
and the lower bound series diverges as soon as the radii satisfy $R_j \ge j$.
\end{proof}
\end{appendix}

\bigskip
\footnotesize
\noindent\textit{Acknowledgements.} We would like to thank Diego C\'{o}rdoba, Anne Bronzi, \'{A}ngel Castro, Sara Daneri, Luis Guijarro and Yann Brenier for various discussions on topics related to the paper.

\end{document}